\documentclass{amsart}

\usepackage{amsmath}
\usepackage{amsthm}
\usepackage{amssymb}
\usepackage{amsfonts}
\usepackage[utf8]{inputenc}

\theoremstyle{plain}
\newtheorem{theorem}{Theorem}[section]
\newtheorem{lemma}[theorem]{Lemma}

\newtheorem{corollary}[theorem]{Corollary}
\newtheorem{proposition}[theorem]{Proposition}

\newtheorem{question}[theorem]{Question}

\theoremstyle{definition}\newtheorem{definition}[theorem]{Definition}

\theoremstyle{definition}
\theoremstyle{definition}\newtheorem{remark}[theorem]{Remark}

\numberwithin{equation}{section}

\newcommand{\R}{\mathbb{R}}
\newcommand{\Q}{\mathbb{Q}}
\newcommand{\N}{\mathbb{N}}
\newcommand{\Z}{\mathbb{Z}}

\newcommand{\rank}{\qopname\relax o{rk}}

\newcommand{\iB}{\mathcal{B}}
\newcommand{\iF}{\mathcal{F}}
\newcommand{\iU}{\mathcal{U}}
\newcommand{\iV}{\mathcal{V}}
\newcommand{\al}{\alpha}
\newcommand{\la}{\lambda}
\newcommand{\om}{\omega}
\newcommand{\si}{\sigma}
\newcommand{\eps}{\varepsilon}
\newcommand{\RR}{\mathbb{R}}
\newcommand{\bd}{\begin{definition}}
\newcommand{\ed}{\end{definition}}
\newcommand{\solc}{\mathop{sc}}
\newcommand{\lab}{\label}

\begin{document}

\title{Ranks on the Baire class $\xi$ functions}

\author[M\'arton Elekes]{M\'arton Elekes$^\ast$}
\thanks{$^\ast$Partially supported by the
Hungarian Scientific Foundation grant no.~83726.}

\author[Viktor Kiss]{Viktor Kiss$^\dag$}
\thanks{$^\dag$Partially supported by the
Hungarian Scientific Foundation grant no.~105645.}

\author[Zolt\'an Vidny\'anszky]{Zolt\'an Vidny\'anszky$^\ddag$}
\thanks{$^\ddag$Partially supported by the
Hungarian Scientific Foundation grant no.~104178.}

\subjclass[2010]{Primary 26A21; Secondary 03E15, 54H05.}
\keywords{Baire class $\xi$ functions, ordinal ranks, 
  descriptive set theory.}

\begin{abstract}
In 1990 Kechris and Louveau developed the theory of three very natural ranks on the Baire class $1$ functions. A rank is a function assigning countable ordinals to certain objects, typically measuring their complexity. We extend this theory to the case of Baire class $\xi$ functions, and generalize most of the results from the Baire class $1$ case. We also show that their assumption of the compactness of the underlying space can be eliminated. As an application, we solve a problem concerning the so called solvability cardinals of systems of difference equations, arising from the theory of geometric decompositions. We also show that certain other very natural generalizations of the ranks of Kechris and Louveau surprisingly turn out to be bounded in $\omega_1$.
Finally, we prove a general result showing that all ranks satisfying some natural properties coincide for bounded functions.
\end{abstract}

\maketitle

\section{Introduction}

A real-valued function defined on a complete metric space is called \emph{Baire class $1$} if it is the pointwise limit of a sequence of continuous functions. It is well-known that a function is of Baire class $1$ iff the inverse image of every open set is $F_\si$ iff there is a point of continuity relative to every non-empty closed set \cite{K}. Baire class $1$ functions play a central role in various branches of mathematics, most notably in Banach space theory, see e.g. \cite{AGR} or \cite{HOR}. A fundamental tool in the analysis of Baire class $1$ functions is the theory of ranks, that is, maps assigning countable ordinals to Baire class $1$ functions, typically measuring their complexity. In their seminal paper \cite{KL}, Kechris and Louveau systematically investigated three very important ranks on the Baire class $1$ functions. We will recall the definitions in Section \ref{s:KL} below, and only note here that they correspond to above three equivalent definitions of Baire class 1 functions. One can 
easily see 
that the theory has no straightforward generalization to the case of Baire class $\xi$ functions. (Recall that $f$ is of \emph{Baire class $\xi$} if there exist sequences $\xi_n < \xi$ and $f_n$ such that $f_n$ is of Baire class $\xi_n$ and $f_n \to f$ pointwise.)

Hence the following very natural but somewhat vague question arises.

\begin{question}
\lab{q:KL}
Is there a natural extension of the theory of Kechris and Louveau to the case of Baire class $\xi$ functions?
\end{question}

There is actually a very concrete version of this question that was raised by Elekes and Laczkovich in \cite{EL}. In order to be able to formulate this we need some preparation. For $\theta, \theta' < \om_1$ let us define the relation 
  $\theta \lesssim \theta'$ if $\theta' \le \om^\eta \implies \theta \le \om^\eta$ for every $1 \le \eta < \om_1$ (we use ordinal exponentiation here).
  Note that $\theta \le \theta'$ implies $\theta \lesssim \theta'$, while $\theta \lesssim \theta'$, $\theta'>0$ implies $\theta \le \theta' \cdot \om$.
  We will also use the notation $\theta \approx \theta'$ if $\theta \lesssim \theta'$ and 
  $\theta' \lesssim \theta$. Then $\approx$ is an equivalence relation.
  Let us denote the set of Baire class $\xi$ functions defined on $\R$ by
  $\iB_\xi(\R)$. The characteristic function of a set
  $H$ is denoted by $\chi_H$. A set is called \emph{perfect} if it is closed and has no isolated points. Define the translation map $T_t : \R \to \R$ by $T_t(x) = x+t$ for every $x \in \R$.

\begin{question}(\cite[Question 6.7]{EL})
\lab{q:EL}
Is there a map $\rho : \iB_\xi(\R) \to \om_1$ such that
\begin{itemize}
\item
$\rho$ is unbounded in $\om_1$, moreover,
for every non-empty perfect set $P \subseteq \R$ and ordinal $\zeta <
\omega_1$ there is a function $f \in \iB_\xi(\R)$ such that $f$ is $0$
outside of $P$ and $\rho(f) \ge \zeta$,
\item
$\rho$ is \textit{translation-invariant}, i.e., $\rho(f \circ T_t) = \rho(f)$ for every $f \in \iB_\xi(\R)$ and $t \in \R$,
\item
$\rho$ is \textit{essentially linear}, i.e., $\rho(cf) \approx \rho(f)$ and $\rho(f + g) \lesssim \max\{\rho(f), \rho(g)\}$ for every $f,g \in \iB_\xi(\R)$ and $c \in \R \setminus \{0\}$,
\item
$\rho(f \cdot \chi_F) \lesssim \rho(f)$ for every closed set $F \subseteq \R$ and $f \in \iB_\xi(\R)$?
\end{itemize}
\end{question}

The problem is not formulated in this exact form in \cite{EL}, but a careful examination of the proofs there reveals that this is what they need for their results to go through. Actually, there are numerous equivalent formulations, for example we may simply replace $\lesssim$ by $\le$ (indeed, just replace $\rho$ satisfying the above properties by $\rho'(f) = \min\{\om^\eta : \rho(f) \le \om^\eta\}$). However, it turns out, as it was already also the case in \cite{KL}, that $\lesssim$ is more natural here.
  
Their original motivation came from the theory of paradoxical geometric decompositions (like the Banach-Tarski paradox, Tarski's problem of circling the square, etc.). It has turned out that the solvability of certain systems of difference equations plays a key role in this theory.

\bd
Let $\R^\R$ denote the set of functions from $\R$ to $\R$. A \emph{difference operator} is a mapping $D:\RR^\RR\to\RR^\RR$ of the form
\[
(Df)(x)=\sum_{i=1}^n a_i f(x+b_i),
\]
where $a_i$ and $b_i$ are fixed real numbers.
\ed

\bd
A \emph{difference equation} is a functional equation
\[
Df=g,
\]
where $D$ is a difference operator,
$g$ is a given function and $f$ is the unknown. 
\ed

\bd
A \emph{system of difference equations} is 
\[
D_i f=g_i \ \ (i\in I),
\]
where $I$ is an arbitrary set of indices. 
\ed

It is not very hard to show that a system of difference
equations is solvable iff every \emph{finite} subsystem is solvable. 
But if we are interested in continuous solutions then this result is
no longer true. However, if every \emph{countable} subsystem of a 
system has a continuous solution the the whole system has a
continuous solution as well. This motivates the following definition,
which has turned out to be a very useful tool for finding necessary
conditions for the existence of certain solutions.

\bd
Let $\mathcal{F}\subset \RR^\RR$ be a class of real functions. The
\emph{solvability cardinal} of $\mathcal{F}$ is the minimal cardinal
$\solc(\mathcal{F})$ with the property that if every subsystem of size
less than $\solc(\mathcal{F})$ of a system of difference equations has a
solution in $\mathcal{F}$ then the whole system has a solution in
$\mathcal{F}$.
\ed

It was shown in \cite{EL} that the behavior of $\solc (\iF )$ is rather erratic.
For example, $\solc (\text{polynomials})=3$ but $\solc (\text{trigonometric
polynomials})=\om _1$,
$\solc (\{ f: f\ \text{is continuous}\}) = \om_1$ but $\solc (\{ f: f\ 
\text{is Darboux}\}) =(2^\om )^+$, and $\solc (\RR^\RR )=\om$.

It is also proved in their paper that $\om_2 \le \solc ({\{ f: f\ \text{is Borel}\}}) \le (2^\om )^+$,
therefore if we assume the Continuum Hypothesis then $\solc ({\{ f: f\ \text{is Borel}\}}) = \om_2$.
Moreover, they obtained that $\solc(\iB_\xi) \le (2^\om )^+$ for every $2 \le \xi < \om_1$, 
and asked if $\om_2 \le \solc(\iB_\xi)$. They noted that a positive answer to Question
\ref{q:EL} would yield a positive answer here.

For more information on the connection between ranks, solvability cardinals, systems of 
difference equations, liftings, and paradoxical decompositions consult \cite{EL}, \cite{L}, \cite{L1} and the references therein. 

In order to be able to answer the above questions we need to address one more problem. This is slightly unfortunate for us, but Kechris and Louveau have only worked out their theory in compact metric spaces, while it is really essential for our purposes to be able to apply the results in arbitrary Polish spaces.

\begin{question}
\lab{q:KL2}
Does the theory of Kechris and Louveau generalize from compact metric spaces to arbitrary Polish spaces? 
\end{question}

Now we describe our results and say a few words about the organization of the paper. First we review the results of Kechris and Louveau in quite some detail in Section \ref{s:KL}, and also answer Question \ref{q:KL2} in the affirmative. Most of the results in this section are not considered to be new, we only have to check that the proofs in \cite{KL} work in non-compact Polish spaces as well. A notable exception is Theorem \ref{a=b=g_bounded} stating that the three ranks essentially coincide for bounded Baire class $1$ functions, since our highly non-trivial proof for the case of general Polish spaces required completely new ideas. 
Next, in Section \ref{s:neg}, we propose numerous very natural ranks on the Baire class $\xi$ functions that surprisingly turn out to be bounded in $\om_1$! Then we answer Question \ref{q:KL} and Question \ref{q:EL} in the affirmative in Section \ref{s:pos}. We actually define four ranks on every $\iB_\xi$, but two of these turn out to be essentially equal, and the resulting three ranks are very good analogues of the original ranks of Kechris and Louveau. We are actually able to generalize most of their results to these new ranks. As a corollary, we also obtain that $\om_2 \le \solc(\iB_\xi)$, 
and hence if we assume the Continuum Hypothesis then $\solc (\iB_\xi) = \om_2$ for every $2 \le \xi < \om_1$. 

In Section \ref{s:uni} we prove that if a rank has certain natural
properties then it 
coincides with $\alpha,\beta$ and $\gamma$ on the bounded Baire class $1$
functions. We also indicate how one could generalize this to the bounded Baire class $\xi$ case. 

Finally, we collect the open questions in Section \ref{s:open}.

\section{Preliminaries}
  
  Most of the following notations and facts can be found in \cite{K}.

  Throughout the paper, let $(X, \tau)$ be an uncountable \emph{Polish} 
  space, that is, a separable and completely metrizable topological space. We denote a compatible, complete metric for $(X, \tau)$ 
  by $d$. A \emph{Polish group} is a topological group whose topology is Polish.
  
  ${\boldsymbol \Sigma^0_\xi}$, ${\boldsymbol \Pi^0_\xi}$ and 
  ${\boldsymbol \Delta^0_\xi}$ stand for the \emph{$\xi$th additive, 
  multiplicative and ambiguous classes} of the Borel hierarchy. 
  We say that a set $H$ is \emph{ambiguous} if 
  $H \in {\boldsymbol \Delta^0_2}$.
  
If $\tau'$ is a topology on $X$ then we denote the family of real valued
functions defined on $X$ that are of Baire class $\xi$ with respect to $\tau'$
by $\mathcal{B}_\xi(\tau')$. In particular, $\mathcal{B}_\xi =
\mathcal{B}_\xi(\tau)$. If $Y$ is another Polish space (whose
topology is clear from the context) then we also use the notation
$\mathcal{B}_\xi(Y)$ for the family of Baire class $\xi$ functions defined on
$Y$. Similarly, ${\boldsymbol \Sigma^0_\xi} (\tau')$ and ${\boldsymbol
  \Sigma^0_\xi} (Y)$ are both the set of ${\boldsymbol \Sigma^0_\xi}$ subsets,
with respect to $\tau'$, and in $Y$, respectively. We use the analogous notation
for all the other pointclasses.
  
  If $Y$ is a Polish space then a subset $P \subseteq Y$ is 
  \emph{perfect} if it is closed and has no isolated points. 
  A non-empty perfect subset of a Polish space with the subspace 
  topology is an uncountable Polish space. 
  
  For a real valued function $f$ on $X$ and a real number $c$, 
  we let 
  $\{f < c\} = \{ x \in X  : f(x) < c \}$. We use the notations 
  $\{f > c\}$, $\{f \le c\}$, $\{f \ge c\}$ and $\{f \neq c\}$ 
  analogously.
  
  It is well-known that a function is of Baire class $\xi$ iff the inverse image of every open set is in ${\boldsymbol \Sigma^0_{\xi+1}}$ iff 
  $\{f < c\}$ and $\{f > c\}$ are in ${\boldsymbol \Sigma^0_{\xi+1}}$ for every $c \in \R$.  Moreover, the family of Baire class $\xi$ functions is closed under uniform limits.
   
  For a set $H$ we denote the characteristic function, closure and 
  complement of $H$ by $\chi_H$, $\overline{H}$, and $H^c$, 
  respectively. For a set $H \subseteq X \times Y$ and an element 
  $x \in X$ we denote the $x$-section of $H$ by 
  $H^x = \{ y \in Y : (x, y) \in H\}$. 
  
  If $\mathcal{H}$ is a family of sets then 
  \begin{equation*}
  \mathcal{H}_\sigma = 
    \left\{ \bigcup_{n \in \N} H_n : H_n \in \mathcal{H} \right\} 
  \text{ and } \mathcal{H}_\delta = 
    \left\{ \bigcap_{n \in \N} H_n : H_n \in \mathcal{H} \right\}. 
  \end{equation*}  
  
 For $\theta, \theta' < \om_1$ we use the relation 
  $\theta \lesssim \theta'$ if $\theta' \le \om^\eta \implies \theta \le \om^\eta$ for every $1 \le \eta < \om_1$ (we use ordinal exponentiation here).
  Note that $\theta \le \theta'$ implies $\theta \lesssim \theta'$ and $\theta \lesssim \theta'$, $\theta'>0$ implies $\theta \le \theta' \cdot \om$. 
  We write $\theta \approx \theta'$ if $\theta \lesssim \theta'$ and 
  $\theta' \lesssim \theta$. Then $\approx$ is an equivalence relation.
  For every ordinal $\theta$ we have $2\theta < \theta + \omega$, and since $\om^\eta$ is a limit
ordinal for every $\eta \ge 1$ we obtain that $2\theta \approx \theta$ for every ordinal $\theta$.
   
A rank $\rho : \iB_\xi \to \om_1$ is called \emph{additive} if $\rho(f + g) \le \max\{\rho(f), \rho(g)\}$ for every $f,g \in \iB_\xi$. It is called \emph{linear} if it is additive and $\rho(cf) = \rho(f)$ for every $f \in \iB_\xi$ and $c \in \R \setminus \{0\}$. If $X$ is a Polish group then the left and right translation operators are defined as $L_{x_0}(x) = x_0 \cdot x$ $(x \in X)$ and $R_{x_0}(x) = x \cdot x_0$ $(x \in X)$. A rank $\rho : \iB_\xi \to \om_1$ is called \emph{translation-invariant} if $\rho(f \circ L_{x_0}) = \rho(f \circ R_{x_0}) = \rho(f)$ for every $f \in \iB_\xi$ and $x_0 \in X$. We say that it is \emph{essentially additive}, \emph{essentially linear}, and \emph{essentially translation-invariant} if the corresponding inequalities and equations hold with $\lesssim$ and $\approx$. Moreover, $\rho$ is additive, essentially additive etc. \emph{for bounded functions}, if the corresponding relations hold whenever $f$ and $g$ are bounded.
   
    Let $(F_{\eta})_{\eta < \la}$ be a (not necessarily strictly) decreasing sequence of sets. Let us assume that $F_0 = X$ and that the sequence is \emph{continuous}, that     is, $F_{\eta} = \bigcap_{\theta < \eta} F_{\theta}$ for every limit $\eta$ and if $\lambda$ is a limit then $\bigcap_{\eta < \lambda} F_\eta = \emptyset$. We also use the convention that $F_{\eta} = \emptyset$ if $\eta \ge \la$.
  We say that a set $H$ is the \emph{transfinite difference of $(F_{\eta})_{\eta < \la}$} if $H = 
  \bigcup_{\begin{subarray}{c} \eta < \lambda \\ \eta \text{ even} \end{subarray}} 
 (F_{\eta} \setminus F_{\eta + 1})$. It is well-known that a set is in ${\boldsymbol \Delta^0_{\xi+1}}$ iff it is a transfinite difference of ${\boldsymbol \Pi^0_\xi}$ sets see e.g.~\cite[22.27]{K}. We have to point out here that the monograph \cite{K} does \emph{not} assume that the decreasing sequences are continuous, but when proving that every set in ${\boldsymbol \Delta^0_{\xi+1}}$ has a representation as a transfinite difference they actually construct continuous sequences, hence this issue causes no difficulty here. 
  
  The set of sequences of length $k$ whose terms are elements of the set $\{0, \dots, n-1\}$ is denoted by $n^k$. 
  For $s \in n^k$ we denote the $i$-th term of $s$ by 
  $s(i)$. If $l \in \{0, \dots, n-1\}$ then $s^\wedge l$ denotes 
  the sequence in $n^{k + 1}$ whose first $k$ terms agree with those of $s$ and whose $k+1$st term is $l$.
  %For a finite sequence $s$ we denote the length of $s$ by $\length(s)$. 

\section{Ranks on the Baire class 1 functions without compactness}
\lab{s:KL}

In this section we summarize some results concerning ranks on the 
Baire class 1 functions, following the work of Kechris and 
Louveau. We do not consider the results in this section as original, 
we basically just carefully check that the results of Kechris and Louveau
hold without the assumption of compactness of $X$. This is inevitable,
since they assumed compactness throughout their paper but we will need
these results in Section \ref{s:pos} for arbitrary Polish spaces.

A notable exception is Theorem \ref{a=b=g_bounded} stating that the three ranks essentially coincide for bounded Baire class $1$ functions.
Since our highly non-trivial proof for the case of general Polish spaces required completely new ideas, we consider this result as original in the non-compact case.

The definitions of the ranks will use the notion of a \emph{derivative 
operation}. 
\begin{definition}
  A \emph{derivative} on the closed subsets of $X$ is a map 
  $D: {\boldsymbol \Pi^0_1}(X) \to {\boldsymbol \Pi^0_1}(X)$ 
  such that $D(A) \subseteq A$ and 
  $A \subseteq B \Rightarrow D(A) \subseteq D(B)$ for every $A, B \in {\boldsymbol \Pi^0_1}(X)$.
\end{definition}
\begin{definition}
  For a derivative $D$ we define the \emph{iterated derivatives} 
  of the closed set $F$ as follows: 
  \begin{align*}
    D^0(F) &= F, \\
    D^{\eta + 1}(F) &= D(D^{\eta}(F)), \\
    D^{\eta}(F) &= \bigcap_{\theta < \eta} D^{\theta}(F) 
      \text{ if $\eta$ is a limit.}
  \end{align*}
\end{definition}
\begin{definition}
  Let $D$ be a derivative. The \emph{rank} of $D$ is the smallest 
  ordinal $\eta$, such that $D^{\eta}(X) = \emptyset$, if such 
  ordinal exists, $\omega_1$ otherwise. 
  We denote the rank of $D$ by $\rank(D)$.
\end{definition}
\begin{remark}
  \label{derivative_countable} 
  In all our applications $D$ satisfies $D(F) \subsetneqq F$ for every 
  non-empty closed set $F$, and since in a Polish space there is no 
  strictly decreasing sequence of closed sets of length $\omega_1$ (see e.g.~\cite[6.9]{K}), 
  the rank of a derivative is always a countable ordinal.
\end{remark}
\begin{proposition}
  \label{two_derivatives}
  If the derivatives $D_1$ and $D_2$ satisfy 
  $D_1(F) \subseteq D_2(F)$ for every closed subset $F \subseteq X$ 
  then $\rank(D_1) \le \rank(D_2)$. 
\end{proposition}
\begin{proof}
  It is enough to prove that $D_1^\eta(X) \subseteq D_2^\eta(X)$ for 
  every ordinal $\eta$. We prove this by transfinite induction on 
  $\eta$. For $\eta = 0$ this is obvious, since 
  $ D_1^0(X) = D_2^0(X) = X$. Now suppose this holds for $\eta$ and 
  we prove it for $\eta + 1$. Since 
  $D_1^\eta(X) \subseteq D_2^\eta(X)$ and $D_1$ is a derivative, we 
  have $D_1(D_1^\eta(X)) \subseteq D_1(D_2^\eta(X))$. Using this 
  observation and the condition of the proposition for the closed 
  set $D_2^\eta(X)$, we have 
  $D_1^{\eta + 1}(X) = D_1(D_1^\eta(X)) \subseteq D_1(D_2^\eta(X)) 
  \subseteq D_2(D_2^\eta(X)) = D_2^{\eta + 1}(X)$.
  
  For limit $\eta$ the claim is an easy consequence of the continuity 
  of the sequences, hence the proof is complete. 
\end{proof}

\begin{proposition}
  \label{ugly_lemma}
  Let $n \ge 1$ and $D$, $D_0, \dots, D_{n - 1}$ be 
  derivative operations on the closed subsets of $X$. 
  Suppose that they satisfy the following conditions 
  for arbitrary closed sets $F$ and $F'$:
  \begin{equation} \label{lemma1}
    D(F) \subseteq \bigcup_{k = 0}^{n - 1} D_k(F), 
  \end{equation}
  \begin{equation} \label{lemma2}
    D(F \cup F') \subseteq 
    D(F) \cup D(F').
  \end{equation}
  Then for these derivatives
  \begin{equation}
    \rank (D) \lesssim 
    \max_{k < n} \rank (D_k). 
  \end{equation}
\end{proposition}

\begin{proof}
  We will prove by induction on $\eta$ that 
  \begin{align}
    \label{induction1}
    D^{\omega^\eta}(F) \subseteq 
      \bigcup_{k = 0}^{n - 1} D_k^{\omega^\eta}(F)
  \end{align}
  for every closed set $F$. It is easy to see that proving \eqref{induction1} is enough, 
  since if $\eta$ is an ordinal satisfying 
  $\rank(D_k) \le \omega^\eta$ for 
   	  every $k < n$ then we have 
   	  $\rank(D) \le \omega^\eta$.
  
  Now we prove \eqref{induction1}. The case $\eta = 0$ is exactly \eqref{lemma1}.
  For limit $\eta$ the statement is obvious, since the sequences 
  are decreasing and continuous. Hence, it remains to prove \eqref{induction1} 
  for $\eta + 1$ if it holds for $\eta$. For this it is 
  enough to show that for every $m \in \omega$ 
  \begin{equation}
    \label{induction2}
    D^{\omega^\eta \cdot m \cdot n}(F) \subseteq 
      \bigcup_{k = 0}^{n - 1} D_k^{\omega^\eta \cdot m}(F),
  \end{equation}
  indeed, 
\[
      D^{\omega^{\eta + 1}}(F) = \bigcap_{m \in \om} D^{\omega^\eta \cdot m \cdot n}(F) \subseteq 
      \bigcap_{m \in \om} \left(\bigcup_{k = 0}^{n - 1} D_k^{\omega^\eta \cdot m}(F) \right),
\]  
  hence $x \in D^{\omega^{\eta + 1}}(F)$ implies that without loss of generality $x \in D_0^{\omega^\eta \cdot m}(F)$ for infinitely many $m$, but the sequence $D_0^{\omega^\eta \cdot m}(F)$ is decreasing, hence $x \in \bigcap_{m \in \om} D_0^{\omega^\eta \cdot m}(F) = D_0^{\omega^{\eta + 1}}(F)$. 
  
  Now we prove \eqref{induction2}. Let $F_\emptyset = F$, and for $m \in \N$, $s \in n^m$ and $k< n$ let 
  \begin{equation*}
    F_{s^\wedge k} = 
      D_k^{\omega^\eta}(F_s).
  \end{equation*}
  
  It is enough that for $m \ge 1$ 
  \begin{equation}
    \label{induction3}
    D^{\omega^\eta \cdot m}(F) \subseteq 
    \bigcup_{s \in n^{m}} F_s,
  \end{equation}
  since it is easy to see that 
  \begin{equation*}
    \bigcup_{s \in n^{m \cdot n}} F_s \subseteq 
      \bigcup_{k = 0}^{n - 1} \bigcup \{F_s : s \in n^{m \cdot n} \text{ and } 
      |\{i : s(i) = k \}| \ge m \} ,
  \end{equation*}
  yielding \eqref{induction2}, as
  \begin{equation*}
    \bigcup \{F_s : s \in n^{m \cdot n} \text{ and } 
      |\{i : s(i) = k \}| \ge m \} \subseteq 
    D_k^{\omega^\eta \cdot m}(F).
  \end{equation*}
  
  It remains to prove \eqref{induction3}  by induction on $m$.  
  For $m = 1$, this is only the induction hypothesis of 
  \eqref{induction1} for $\eta$. By supposing \eqref{induction3} 
  for $m$, we have 
  \begin{equation*}
  \begin{split}
    D^{\omega^\eta \cdot (m + 1)}(F) &=  
      D^{\omega^\eta} \left( 
      D^{\omega^\eta \cdot m}(F) \right) 
      \subseteq D^{\omega^\eta} \left( 
      \bigcup_{s \in n^{m}} F_s \right) \subseteq \\
    &\subseteq \bigcup_{s \in n^{m}} 
      D^{\omega^\eta}(F_s)
      \subseteq \bigcup_{s \in n^{m + 1}} F_s ,
  \end{split}
  \end{equation*}
  where we used \eqref{lemma2} $\omega^\eta$ many times for the 
  second containment, and for the last one we used the induction
  hypothesis, that is \eqref{induction1} for $\eta$. 
  This finishes the proof.
\end{proof}

\subsection{The separation rank}
  This rank was first introduced by Bourgain \cite{B}.
  \begin{definition}
    Let $A$ and $B$ be two subsets of $X$. We associate a derivative 
    with them by 
    \begin{equation}
      \label{alpha_derivative}
      D_{A, B}(F) = \overline{F \cap A} \cap \overline{F \cap B}.
    \end{equation}
  \end{definition}
  It is easy to see that $D_{A, B}(F)$ is closed, 
  $D_{A, B}(F) \subseteq F$ and 
  $D_{A, B}(F) \subseteq D_{A, B}(F')$ for every pair of sets $A$ and 
  $B$ and every pair of closed sets 
  $F \subseteq F'$, hence $D_{A, B}$ is a derivative. We use 
  the notation $\alpha(A, B) = \rank( D_{A, B} )$. 
  \begin{definition}
    The \emph{separation rank} of a Baire class 1 function $f$ 
    is defined as 
    \begin{equation} 
      \label{alpha_def}
      \alpha(f) = 
        \sup_{\begin{subarray}{c} p < q \\ p, q \in \Q \end{subarray}}
        \alpha(\{f \le p\}, \{f \ge q \}).
    \end{equation}
  \end{definition}
  
  \begin{remark}
    \label{alpha_def_real}
   Actually, 
    \begin{equation*}
      \alpha(f) = 
        \sup_{\begin{subarray}{c} x < y \\ x, y \in \R \end{subarray}}
        \alpha(\{f \le x\}, \{f \ge y \}),
    \end{equation*}
    since if $x < p < q < y$ then 
    $\alpha(\{f \le x\}, \{f \ge y\}) \le 
    \alpha(\{f \le p\}, \{f \ge q\})$, since any set 
    $H \in {\boldsymbol \Delta^0_2}(X)$ separating the level sets 
    $\{f \le p\}$ and  $\{f \ge q\}$ also separates 
    $\{f \le x\}$ and  $\{f \ge y\}$.
  \end{remark}

  \begin{proposition}
    \label{alpha_le_omega_1}
    If $f$ is a Baire class 1 function then $\alpha(f) < \omega_1$. 
  \end{proposition}
  \begin{proof}
    From the definition of the rank and Remark \ref{derivative_countable}
    it is enough to prove that for any pair of rational numbers 
    $p < q$ and non-empty closed set $F \subseteq X$, $D_{A, B}(F) \subsetneq F$, 
    where $A = \{f \le p\}$ and $B = \{f \ge q\}$. 
    Since $f$ is of Baire class 1, it has a point of continuity 
    restricted to $F$, hence $A$ and $B$ cannot be both dense in $F$. 
    Consequently, $D_{A, B}(F) = \overline{F \cap A} \cap \overline{F \cap B} 
    \subsetneq F$, proving the proposition. 
  \end{proof}

  Next we prove that $\alpha(A, B) < \omega_1$ iff $A$ and $B$ can be 
  separated by a transfinite difference of closed sets.
 
  \begin{definition}
    If the sets $A$ and $B$ can be separated by a transfinite 
    difference of closed sets then let 
    $\alpha_1(A, B)$ denote the length of the shortest such sequence,
    otherwise let $\alpha_1(A, B) = \omega_1$. 
    We define the \emph{modified separation rank} of a 
    Baire class 1 function $f$ as 
    \begin{equation}
      \label{alpha_1_def}
      \alpha_1(f) = 
        \sup_{\begin{subarray}{c} p < q \\ p, q \in \Q \end{subarray}}
        \alpha_1(\{f \le p\}, \{f \ge q \}).
    \end{equation}
  \end{definition}

  \begin{proposition}
    \label{alpha_1_equals_alpha_for_sets}
    Let $A$ and $B$ two subsets of $X$. Then
    $$
      \alpha(A, B) \leq \alpha_1(A, B) \leq 2\alpha(A, B)
      \text{, hence $\alpha(A, B) \approx \alpha_1(A, B)$.}
    $$
  \end{proposition}
 \begin{proof}
    For the first inequality we can assume that 
    $\alpha_1(A, B) < \omega_1$, so $A$ and $B$ can be separated by a 
    transfinite difference of closed sets. Let $(F_{\eta})_{\eta<\lambda}$ be such a 
    sequence, where $\lambda = \alpha_1(A, B)$. Now we have 
    $$
      A \subseteq \bigcup_{
        \begin{subarray}{c} \eta < \lambda \\ \eta 
          \text{ even} \end{subarray}} 
        (F_{\eta} 
      \setminus F_{\eta + 1}) \subseteq B^c.
    $$
    It is enough to prove that 
    $D_{A, B}^{\eta}(X) \subseteq F_{\eta}$ for every $\eta$. We prove this by 
    induction. 
    For $\eta = 0$ this is obvious, since $D_{A, B}^0(X) = F_0 = X$. 
    
    Now suppose that $D_{A, B}^{\eta}(X) \subseteq F_{\eta}$. We 
    show that $D_{A, B}^{\eta + 1}(X) = 
      \overline{D_{A, B}^{\eta}(X) \cap A} \cap 
      \overline{D_{A, B}^{\eta}(X) \cap B} \subseteq 
      F_{\eta + 1}$. 
    If $\eta$ is even then 
    $$ 
      D^{\eta}_{A, B}(X) \setminus F_{\eta + 1} \subseteq F_{\eta} \setminus F_{\eta + 1} \subseteq B^c,
    $$
    hence 
    $D^{\eta}_{A, B}(X) \cap B \subseteq F_{\eta + 1}$. Since 
    $F_{\eta + 1}$ is closed, we obtain 
    $\overline{D^{\eta}_{A, B}(X) \cap B} \subseteq F_{\eta + 1}$, 
    hence $D_{A, B}^{\eta + 1} \subseteq F_{\eta + 1}$. 
    If $\eta$ is odd then $F_{\eta} \setminus F_{\eta + 1}$ is disjoint from $\bigcup_{
        \begin{subarray}{c} \eta < \lambda \\ \eta 
          \text{ even} \end{subarray}} 
        (F_{\eta} 
      \setminus F_{\eta + 1})$, hence 
    $F_{\eta} \setminus F_{\eta + 1} \subseteq A^c$, and an argument analogous to the above one yields  
    $\overline{D^{\eta}_{A, B}(X) \cap A} \subseteq F_{\eta + 1}$, 
    hence $D_{A, B}^{\eta + 1} \subseteq F_{\eta + 1}$. 
    
    If $\eta$ is limit and 
    $D_{A, B}^{\theta}(X) \subseteq F_{\theta}$ 
    for every $\theta < \eta$ then 
    $D_{A, B}^{\eta}(X) \subseteq F_{\eta}$ because the sequences 
    $D_{A, B}^{\eta}(X) $ and $ F_{\eta}$ are continuous. 

    For the second inequality we suppose that 
    $\alpha(A, B) < \omega_1$, that is, the sequence $D_{A, B}^{\eta}(X)$ terminates at the empty set at some countable ordinal.
    Let 
    $$
      F_{2\eta} = D_{A, B}^{\eta}(X), \quad 
      F_{2\eta + 1} = \overline{D_{A, B}^{\eta}(X) \cap B}. 
    $$
    Clearly, $F_0 = X$ and $F_{2\eta} \supseteq F_{2\eta+1}$ for every $\eta$. It is easily seen from the definition of $D_{A, B}^{\eta+1}(X)$ that $F_{2\eta+1} \supseteq F_{2\eta+2}$ for every $\eta$. Moreover, the sequence $F_{2\eta} = D_{A, B}^{\eta}(X)$ is continuous. This implies that the sequence formed by the $F_{\eta}$'s is decreasing and continuous.
    
    Now we show that the transfinite difference of this sequence separates $A$ and $B$.
  
    Every ring of the form $F_{2\eta} \setminus F_{2\eta + 1}$ 
    is disjoint from $B$, so we only need to prove that $A$ is contained 
    in the union of these rings. We show that $A$ is disjoint from the 
    complement of this union by proving that
    $$
      \left( F_{2\eta + 1} \setminus F_{2\eta + 2}\right) \cap A = 
      \left( \overline{D_{A, B}^{\eta}(X) \cap B} \setminus 
         D_{A, B}^{\eta + 1}(X) \right) \cap A = \emptyset
    $$
    for every $\eta$. From the 
    definition of the derivative,
    $D_{A, B}^{\eta + 1}(X) = 
      \overline{D_{A, B}^{\eta}(X) \cap A} \cap 
      \overline{D_{A, B}^{\eta}(X) \cap B}$.
    Using the fact that $D_{A, B}^{\eta}(X)$ is closed, for a point 
    $x \in A \cap \overline{D_{A, B}^{\eta}(X) \cap B}$ we have 
    $x \in \overline{D_{A, B}^{\eta}(X) \cap A}$, hence 
    $x \in D_{A, B}^{\eta + 1}(X)$.
  \end{proof}
  
  \begin{remark}
    It is claimed in \cite{KL} that if $X$ is compact and
    $\alpha(A, B) = \lambda + n$ with $\lambda$ limit and $0 < n \in \om$ then 
    $\alpha_1(A, B)$ is either $\lambda + 2n$ or $\lambda + 2n - 1$. However, this does not seem to be true. 
    For a counterexample, let $X$ be the $2n + 1$-dimensional cube in 
    $\R^{2n + 1}$. 
    Let $A = (F_0 \setminus F_1) \cup (F_2 \setminus F_3) \cup 
    \dots \cup (F_{2n} \setminus F_{2n + 1})$, where $F_i$ is a 
    $(2n + 1 - i)$-dimensional face of $X$, and 
    $F_{i + 1} \subseteq F_i$ for $i \le 2n$. Let $B = X \setminus A$. 
    The definition of $A$ shows that $\alpha_1(A, B) \le 2n + 2$. 
    
    Now $D^0_{A, B}(X) = X = F_0$, and by induction, 
    $D^i_{A, B}(X) = F_i$ for $0 \le i \le 2n + 1$, since 
    $D^i_{A, B}(X) = D(D^{i - 1}_{A, B}(X)) = D_{A, B}(F_{i - 1}) = 
      \overline{F_{i - 1} \cap A} \cap \overline{F_{i - 1} \cap B} = 
      F_i$. 
    Now we have $D^{2n + 2}_{A, B}(X) = 
      D_{A, B}(D^{2n + 1}_{A, B}(X)) = D_{A, B}(F_{2n + 1}) = 
      \emptyset$, proving that in this case $\alpha(A, B) = 2n + 2$. 
    Using Proposition \ref{alpha_1_equals_alpha_for_sets} this shows 
    that $\alpha_1(A, B) = \alpha(A, B) = 2n + 2$.
  \end{remark}
  We leave the proof of the following corollary to the reader.
  \begin{corollary}
    \label{alpha_1_equals_alpha}
    If $f$ is a Baire class 1 function then 
    $$
      \alpha(f) \leq \alpha_1(f) \leq 2\alpha(f) \text{, hence }
      \alpha(f) \approx \alpha_1(f).
    $$      
  \end{corollary}

  \begin{corollary}
    If $f$ is a Baire class 1 function then $\alpha_1(f) < \omega_1$. 
  \end{corollary}
  \begin{proof}
    It is an easy consequence of the previous corollary and Proposition 
    \ref{alpha_le_omega_1}.
  \end{proof}

\subsection{The oscillation rank}
  This rank was investigated by numerous authors, see e.g. \cite{HOR}.
  
  First, we define the oscillation of a function, then turn to 
  the oscillation rank.
  
  \begin{definition}
    The \emph{oscillation} of a function $f:X \to \R$ at a point 
    $x \in X$ restricted to a closed set $F \subseteq X$ is
    \begin{equation}
      \label{beta_osc}
      \omega(f, x, F) = \inf \left\{ \sup_{x_1, x_2 \in U \cap F} 
        |f(x_1) - f(x_2)|: \text{$U$ open, $x \in U$} \right\}. 
    \end{equation}
  \end{definition}
  \begin{definition}
	  For each $\varepsilon > 0$ consider the derivative defined by
	  \begin{equation}
	    \label{beta_derivative}
	    D_{f, \varepsilon}(F) = 
	      \left\{ x \in F : \omega(f, x, F) \ge \varepsilon \right\}.
	  \end{equation}
	\end{definition}
	It is obvious that $D_{f, \varepsilon}(F)$ is closed, 
	$D_{f, \varepsilon}(F) \subseteq F$ and 
	$D_{f, \varepsilon}(F) \subseteq D_{f, \varepsilon}(F')$ for every 
	function $f : X \to \R$, every $\varepsilon > 0$ and every pair of 
	closed sets $F \subseteq F'$, hence $D_{f, \varepsilon}$ is a 
	derivative. Let us denote the rank of $D_{f, \varepsilon}$ by 
	$\beta(f, \varepsilon)$. 
	\begin{definition}
	  The \emph{oscillation rank} of a function $f$ is 
	  \begin{equation}
	    \label{beta_def}
	    \beta(f) = \sup_{\varepsilon > 0} \beta(f, \varepsilon).
	  \end{equation}
  \end{definition}

  \begin{proposition}
    If $f$ is a Baire class 1 function then $\beta(f) < \omega_1$.
  \end{proposition}
  \begin{proof}
    Using Remark \ref{derivative_countable}, it is enough to prove 
    $D_{f, \varepsilon}(F) \subsetneq F$ for every $\varepsilon > 0$ 
    and every non-empty closed set $F \subseteq X$. 
    And this is easy, since $f$ restricted to $F$ is continuous at a 
    point $x \in F$, and thus $x \not \in D_{f, \varepsilon}(F)$, 
    hence $D_{f, \varepsilon}(F) \subsetneq F$. 
  \end{proof}

\subsection{The convergence rank}
  Now we turn to the convergence rank following Zalcwasser \cite{Z} 
  and Gillespie and Hurwitz \cite{GH}. 
  \begin{definition}
    Let $(f_n)_{n \in \N}$ be a sequence of real valued continuous functions 
    on $X$. The \emph{oscillation} of this sequence at a point $x$ 
    restricted to a closed set $F \subseteq X$ is 
    \begin{equation}
      \label{gamma_osc}
       \omega((f_n)_{n \in \N}, x, F) = 
        \inf_{
          \begin{subarray}{c} x \in U \\ \text{$U$ open} 
          \end{subarray}} 
        \inf_{N \in \N} 
        \sup \left\{ |f_m(y) - f_n(y)| : 
        n, m \ge N,\; y \in U \cap F \right\}.     
    \end{equation}
  \end{definition}
  
  \begin{definition}
	  Consider a sequence $(f_n)_{n \in \N}$ of real valued continuous 
	  functions, and for each $\varepsilon > 0$, define a derivative as 
	  \begin{equation}
      \label{gamma_derivative}
	    D_{(f_n)_{n \in \N}, \varepsilon}(F) = \left\{ x \in F : 
	      \omega((f_n)_{n \in \N}, x, F) \ge \varepsilon \right\}.
	  \end{equation}
	\end{definition}
  It is easy to see that $D_{(f_n)_{n \in \N}, \varepsilon}(F)$ is 
  closed, 
  $D_{(f_n)_{n \in \N}, \varepsilon}(F) \subseteq F$ and 
  $D_{(f_n)_{n \in \N}, \varepsilon}(F) \subseteq 
    D_{(f_n)_{n \in \N}, \varepsilon}(F')$ 
  for every sequence of continuous functions $(f_n)_{n \in \N}$, 
  every $\varepsilon > 0$ and every pair of closed sets 
  $F \subseteq F'$, hence $D_{(f_n)_{n \in \N}, \varepsilon}$ is a 
  derivative.
	Let us denote the rank of $D_{(f_n)_{n \in \N}, \varepsilon}$ by 
	$\gamma((f_n)_{n \in \N}, \varepsilon)$. 
  \begin{definition}
  	  For a Baire class 1 function $f$ let the \emph{convergence rank} 
    of $f$ be defined by
    \begin{equation}
      \label{gamma_def}    
	    \gamma(f) = 
	      \min \left\{
	        \sup_{\varepsilon > 0} \gamma((f_n)_{n \in \N}, \varepsilon) : 
	        \forall n \text{ $f_n$ is continuous and $f_n \to f$ 
	        pointwise} 
	      \right\}.
    \end{equation}

  \end{definition}

  \begin{proposition}
    If $f$ is a Baire class 1 function then $\gamma(f) < \omega_1$.
  \end{proposition}
  \begin{proof}
    It suffices to show that  
    $D_{(f_n)_{n \in \N}, \varepsilon}(F) \subsetneq F$ for every 
    $\varepsilon > 0$, every non-empty closed set $F \subseteq X$ and every
    sequence of pointwise convergent continuous functions $(f_n)_{n \in \N}$. 
    Suppose the contrary, then for every $N$ the set 
    $G_N = \{x \in F : \exists n,m \ge N \; |f_n(x) - f_m(x)| > \frac{\varepsilon}{2}\}$ is 
    dense in $F$. It is also open in $F$, hence by the Baire category 
    theorem there is a point $x \in F$ such that $x \in G_N$ for every 
    $N \in \N$, hence the sequence $(f_n)_{n \in \N}$ does not converge 
    at $x$, contradicting our assumption.
  \end{proof}

\subsection{Properties of the ranks}
  
  \begin{theorem}
    \label{alpha_le_beta_le_gamma}
    If $f$ is a Baire class 1 function then
    $\alpha(f) \le \beta(f) \le \gamma(f)$.
  \end{theorem}
  \begin{proof}
    For the first inequality, it is enough to prove that for every 
    $p, q \in \Q$, $p < q$  we can find $\varepsilon > 0$ such 
    that $\alpha(\{f \le p\}, \{f \ge q\}) \le \beta(f, \varepsilon)$. 
    Let $A = \{f \le p\}$, $B = \{f \ge q\}$ and 
    $\varepsilon = p - q$. 
    Using Proposition \ref{two_derivatives} it suffices to show that $D_{A, B}(F) \subseteq D_{f, \varepsilon}(F)$ 
    for every $F \in {\boldsymbol \Pi^0_1}(X)$. 
    If $x \in F \setminus D_{f, \varepsilon}(F)$ then $x$ has a 
    neighborhood $U$  such that $\sup_{x_1, x_2 \in U \cap F} 
    |f(x_1) - f(x_2)| < \eps = p - q$, hence $U$ cannot intersect both 
    $A$ and $B$. 
    So $x \not\in D_{A, B}(F)$, proving the first inequality. 

    For the second inequality, let $(f_n)_{n \in \N}$ be a sequence of 
    continuous functions converging pointwise to a function $f$. 
    It is enough to show that 
    $\beta(f, \varepsilon) \le 
      \gamma((f_n)_{n \in \N}, \varepsilon / 3)$. 
    As in the first paragraph we show that 
    $D_{f, \varepsilon}(F) \subseteq 
      D_{(f_n)_{n \in \N}, \varepsilon / 3}(F)$ 
    for every $F \in {\boldsymbol \Pi^0_1}(X)$.
It is enough to show that if 
    $x \in F \setminus D_{(f_n)_{n \in \N}, \varepsilon / 3}(F)$ then  
    $x \notin D_{f, \varepsilon}(F)$. 
    For such an $x$ there is a neighborhood $U$ of $x$ and an 
    $N \in \N$ such that for all $n, m \ge N$ and 
    $x' \in F \cap U$, $|f_n(x') - f_m(x')| < \varepsilon / 3$. 
    Letting $m \to \infty$ we get 
    $|f_n(x') - f(x')| \le \varepsilon / 3$ for all $n \ge N$ and 
    $x' \in F \cap U$. Let $V \subseteq U$ be a neighborhood of $x$ 
    for which $\sup_V f_N - \inf_V f_N < \varepsilon / 6$. 
    Now for every $x', x'' \in V \cap F$ we have 
    $$
      |f(x') - f(x'')| \le 
        |f_{N}(x') - f_{N}(x'')| + 2\frac{\varepsilon}{3} < \frac56 \eps <
        \varepsilon, 
    $$
    showing that $x \not\in D_{f, \varepsilon}(F)$.
  \end{proof}

  \begin{proposition}
    \label{translation_invariance}
    If $X$ is a Polish group then the ranks $\alpha$, $\beta$ and $\gamma$ are translation 
    invariant.
  \end{proposition}
  \begin{proof}
    Note first that for a Baire class 1 function $f$ and 
    ${x_0} \in X$ the functions $f \circ L_{x_0}$ and $f \circ R_{x_0}$ are also of Baire class 1. 
    Since the topology of a topological group is translation invariant, and the 
    the definitions of the ranks depend only on the topology of the 
    space, the proposition easily follows.
  \end{proof}  
  \begin{theorem}
    \label{ranks_not_bounded}
    The ranks are unbounded in $\omega_1$, actually unbounded already on the characteristic functions.
  \end{theorem}
  We postpone the proof, since later we will prove the more general 
  Theorem \ref{alpha_not_bounded}. 

  \begin{proposition}
    \label{smaller_class}
    If $f$ is continuous then $\alpha(f) = \beta(f) = \gamma(f) = 1$. 
  \end{proposition}
  \begin{proof}
    In order to prove $\alpha(f) = 1$, consider the derivative 
    $D_{\{f \le p \}, \{ f \ge q\}}$, where $p < q$ is a pair of 
    rational numbers. Since the level sets $\{f \le p\}$ and 
    $\{f \ge q\}$ are disjoint closed sets,
    $D_{\{f \le p \}, \{ f \ge q\}}  (X)= \emptyset$.
    
    For $\beta(f) = 1$, note that a continuous function $f$ has 
    oscillation 0 at every point restricted to every set, hence 
    $D_{f, \varepsilon}(X) = \emptyset$ for every $\varepsilon > 0$. 
    
    And finally for $\gamma(f) = 1$ consider the sequence of 
    continuous functions $(f_n)_{n \in \N}$, for which $f_n = f$ for 
    every $n \in \N$. It is easy to see that 
    $\omega((f_n)_{n \in \N}, x, F) = 0$ for every point $x \in X$ 
    and every closed set $F \subseteq X$. Now we have that 
    $D_{(f_n)_{n \in \N}, \varepsilon}(X) = \emptyset$ for every 
    $\varepsilon > 0$, hence $\gamma(f) = 1$. 
  \end{proof}
  
  \begin{theorem}
    
    \label{product_characteristic} 
    If $f$ is a Baire class 1 function and $F \subseteq X$ is closed 
    then $\alpha(f \cdot \chi_F) \le 1 + \alpha(f)$,  
    $\beta(f \cdot \chi_F) \le 1 + \beta(f)$ and 
    $\gamma(f \cdot \chi_F) \le 1 + \gamma(f)$.
  \end{theorem}
  \begin{proof}
    First we prove the statement for the ranks $\alpha$ and $\beta$. 
    Let $D$ be a derivative 
    either of the form $D_{A, B}$ or of the form $D_{f, \varepsilon}$ 
    where $A = \{f \le p\}$ and $B = \{f \ge q\}$ for a pair of 
    rational numbers $p < q$ and $\varepsilon > 0$. 
    Let $\overline{D}$ be the corresponding 
    derivative for the function $f \cdot \chi_F$, i.e.~$\overline{D} = 
    D_{A', B'}$ or $\overline{D} = D_{f \cdot \chi_F, \varepsilon}$, 
    where $A' = \{f \cdot \chi_F \le p\}$ and 
    $B' = \{f \cdot \chi_F \ge q\}$. 

    Since the function 
    $f \cdot \chi_F$ is constant $0$ on the open set $X \setminus F$, 
    it is easy to check that in both cases 
    $\overline{D}(X) \subseteq F$. 
    And since the functions $f$ and $f \cdot \chi_F$ agree on $F$, 
    we have by transfinite induction that 
    $\overline{D}^{1 + \eta}(X) \subseteq D^{\eta}(X)$ for every 
    countable ordinal $\eta$, implying that  
    $\alpha(f \cdot \chi_F) \le 1 + \alpha(f)$ and also 
    $\beta(f \cdot \chi_F) \le 1 + \beta(f)$.
    
    Now we prove the statement for $\gamma$. Let $(f_n)_{n \in \N}$ be 
    a sequence of continuous functions converging pointwise to $f$ 
    with $\sup_{\varepsilon > 0} \gamma((f_n)_{n \in \N}, \varepsilon) 
    = \gamma(f)$. Let $g_n(x) = 1 - \min\{1, n \cdot d(x, F)\}$ and 
    set $f_n'(x) = f_n(x) \cdot g_n(x)$. It is easy to check that 
    for every $n$ the function $f_n'$ is continuous and 
    $f_n' \to f \cdot \chi_F$ pointwise. For every 
    $x \in X \setminus F$ there is a neighborhood of $x$ such that 
    for large enough $n$ the function $f'_n$ is $0$ on this 
    neighborhood, hence 
    $D_{(f_n')_{n \in \N}, \varepsilon}(X) \subseteq F$ for every 
    $\varepsilon > 0$. From this point on 
    the proof is similar to the previous 
    cases, since the sequences of functions $(f_n)_{n \in \N}$ and 
    $(f_n')_{n \in \N}$ agree on $F$, hence, by transfinite induction  
    $D_{(f_n')_{n \in \N}, \varepsilon}^{1 + \eta}(X) \subseteq 
    D_{(f_n)_{n \in \N}, \varepsilon}^{\eta}(X)$ for every 
    $\varepsilon > 0$. From this we have 
    $\gamma((f_n')_{n \in \N}, \varepsilon) 
    \le 1 + \gamma((f_n)_{n \in \N}, \varepsilon)$ for every 
    $\varepsilon > 0$, hence 
    $\gamma(f \cdot \chi_F) \le 1 + \gamma(f)$. 
    Thus the proof of the theorem is complete. 
  \end{proof}
  
  \begin{theorem}
    \label{beta_gamma_sum}
    The ranks $\beta$ and $\gamma$ are essentially linear.
  \end{theorem}
  \begin{proof}
    It is easy to see that $\beta(cf) = \beta(f)$ and $\gamma(cf) = \gamma(f)$ for every $c \in \R \setminus \{0\}$, hence it suffices to show that $\beta$ and $\gamma$ are essentially additive.
   
    First we consider a modification of the definition of the rank $\beta$ as follows. 
    Let $\beta_0$ be the rank obtained by simply replacing $\sup_{x_1, x_2 \in U \cap F} 
        |f(x_1) - f(x_2)|$ in \eqref{beta_osc} by
    $\sup_{x_1 \in U \cap F} 
        |f(x) - f(x_1)|$ in the definition of $\beta$. Clearly, $\beta_0(f, \varepsilon) \le \beta(f, \varepsilon) \le 
    \beta_0(f, \varepsilon / 2)$, hence actually $\beta_0 = \beta$.
    Therefore it is sufficient to prove the theorem for 
    $\beta_0$.
    
    To prove the theorem for $\beta_0$, let 
    $D_0 = D_{f, \varepsilon / 2}$, 
    $D_1 = D_{g, \varepsilon / 2}$ and 
    $D = D_{f + g, \varepsilon}$ (we use here the derivatives defining 
     $\beta_0$). 
    We show that the conditions of Proposition \ref{ugly_lemma} 
    hold for these derivatives. 
    
    For condition \eqref{lemma1}, let 
    $x \in D_{f + g, \varepsilon}(F)$. 
    Since $\omega(f + g, x, F) \ge \varepsilon$, we have 
    $\omega(f, x, F)$ or $\omega(g, x, F) \ge \varepsilon / 2$, 
    hence 
    $x \in D_{f, \varepsilon / 2}(F) \cup D_{g, \varepsilon / 2}(F)$. 
    
    Condition \eqref{lemma2} is similar, let 
    $x \in (F \cup F') \setminus \left( D_{f + g, \varepsilon}(F) \cup 
    D_{f + g, \varepsilon}(F') \right)$. Since 
    $x \not \in D_{f + g, \varepsilon}(F)$, there is a neighborhood 
    $U$ of $x$ with 
    $|(f + g)(x) - (f + g)(x')| < \varepsilon' < \varepsilon$ for 
    $x' \in U \cap F$. And similarly, there is a neighborhood $U'$ 
    with $|(f + g)(x) - (f + g)(x')| < \varepsilon'' < \varepsilon$ 
    for $x' \in U' \cap F'$. Now the neighborhood $U \cap U'$ shows 
    that $\omega({f + g}, x, F \cup F') < \varepsilon$, proving that 
    $x \not \in D_{f + g, \varepsilon}(F \cup F')$. 
    
    The proposition yields that 
    $\beta_0(f + g, \varepsilon) \lesssim \max 
      \{ \beta_0(f, \varepsilon / 2), \beta_0(g, \varepsilon / 2)\}$, 
    hence $\beta_0(f + g) \lesssim \max \{ \beta_0(f), \beta_0(g)\}$. 
    This proves the statement for $\beta_0$, hence for $\beta$. 
    
    For $\gamma$, we do the same, prove the conditions of the 
    proposition for $D_0 = D_{(f_n)_{n \in \N}, \varepsilon / 2}$, 
    $D_1 = D_{(g_n)_{n \in \N}, \varepsilon / 2}$ and 
    $D = D_{(f_n + g_n)_{n \in \N}, \varepsilon}$, and use 
    the conclusion of the proposition to finish the proof. 
    
    For condition \eqref{lemma1}, let 
    $x \in F \setminus 
      \left(D_{(f_n)_{n \in \N}, \varepsilon / 2}(F) \cup 
      D_{(g_n)_{n \in \N}, \varepsilon}(F) \right)$. 
    Now we can choose a common open set $x \in U$  and a common $N \in \N$ 
    such that for all $n, m \ge N$ and $y \in U \cap F$ we have 
    $|f_n(y) - f_m(y)| \le \varepsilon' < \varepsilon / 2$ and 
    $|g_n(y) - g_m(y)| \le \varepsilon' < \varepsilon / 2$ 
    (again, with a common $\varepsilon' < \varepsilon / 2$). 
    But from this we have $|(f_n + g_n)(y) - (f_m + g_m)(y)| \le 
      2\varepsilon' < \varepsilon$ for all $n, m \ge N$ and 
    $y \in U \cap F$, so 
    $x \not \in D_{(f_n + g_n)_{n \in \N}, \varepsilon}(F)$,  
    yielding \eqref{lemma1}.
    
    For \eqref{lemma2} let 
    $x \in (F \cup F') \setminus 
      \left(D_{(f_n + g_n)_{n \in \N}, \varepsilon}(F) 
      \cup D_{(f_n + g_n)_{n \in \N}, \varepsilon}(F') \right)$. 
    For this $x$ we have a neighborhood $U$ of $x$, $N \in \N$ and 
    $\varepsilon' < \varepsilon$, such that 
    $|(f_n + g_n)(y) - (f_m + g_m)(y)| \le \varepsilon'$ for every 
    $n, m \ge N$ and 
    $y \in U \cap F$. Similarly, we can find a neighborhood $U'$, 
    $N' \in \N$ and $\varepsilon'' < \varepsilon$, such that 
    $|(f_n + g_n)(y) - (f_m + g_m)(y)| \le \varepsilon''$ for every 
    $n, m \ge N'$ and $y \in U' \cap F'$. 
    From this, 
    $\omega((f_n + g_n)_{n \in \N}, x, F \cup F') \le 
      \max\{\varepsilon', \varepsilon''\} < \varepsilon$, hence 
    $x \not \in D_{(f_n + g_n)_{n \in \N}, \varepsilon}(F \cup F')$. 

    Therefore the proof of the theorem is complete. 
  \end{proof}
  
  \begin{remark}
    \label{alpha_not_sum}
    The analogous result does not hold for the rank $\alpha$. To see this note first 
    that $\alpha(A, A^c)$ can be arbitrarily large below $\om_1$ when $A$ ranges over ${\boldsymbol \Delta^0_2}(X)$. This is a classical 
    fact and we prove a more general result in Corollary 
    \ref{alpha_unbounded_on_perfect}.
    
    First we check that for every 
    $A \in {\boldsymbol \Delta^0_2}(X)$ the characteristic function $\chi_A$ can be written 
    as the difference of two upper semicontinuous (usc) functions. 
    Indeed, let $(K_n)_{n \in \om}$ and $(L_n)_{n \in \om}$ 
    be increasing sequences of closed sets with $A = \bigcup_n K_n$ 
    and $A^c = \bigcup_n L_n$, and let
    \begin{equation*}
      f_0 = \left\{ 
      \begin{array}{cl} 
        0 &  \text{on $K_0 \cup L_0$,} \\
        -n & \text{on $(K_n \cup L_n) \setminus 
          (K_{n - 1} \cup L_{n - 1})$ for $n \ge 1$} \\
      \end{array}
      \right.
    \end{equation*}
    and 
    \begin{equation*}
      f_1 = \left\{ 
      \begin{array}{cl} 
        0 &  \text{on $L_0$,} \\
        -1 & \text{on $(K_0 \cup L_1) \setminus L_0$,} \\
        -n & \text{on $(K_{n - 1} \cup L_n) \setminus 
          (K_{n - 2} \cup L_{n - 1})$ for $n \ge 2$.} \\
      \end{array}
      \right.
    \end{equation*}
    Then $f_0$ and $f_1$ are usc functions with $\chi_A = f_0 - f_1$. 
    
    Now we complete the remark by showing that $\alpha(f) \le 2$ for every 
    usc function $f$. For $p < q$ let $A = \{f \le p\}$ and 
    $B = \{f \ge q\}$. Then $B$ is closed, so 
    $D_{A, B}(X) = \overline{X \cap A} \cap \overline{X \cap B} = 
    \overline{X \cap A} \cap B \subseteq B$. 
    Hence $D_{A, B}^2(X) \subseteq D_{A, B}(B) = 
      \overline{A \cap B} \cap B = \emptyset \cap B = \emptyset$.
  \end{remark}

  \begin{remark}
  \label{multiplicativity} 
   One can easily deduce from Theorem \ref{beta_gamma_sum} that $\beta(f\cdot g)
   \lesssim \max\{\beta(f),\beta(g)\}$ whenever $f$ and $g$ are \textit{bounded} Baire class $1$ functions, and
   similarly for $\gamma$. However, we do not know if this holds for arbitrary Baire class $1$ functions.
 \end{remark}
 
 \begin{question}
   Are the ranks $\beta$ and $\gamma$ \emph{essentially multiplicative} on the 
   Baire class $1$ functions, that is, does $\beta(f\cdot g)
   \lesssim \max\{\beta(f),\beta(g)\}$ and $\gamma(f\cdot g)
   \lesssim \max\{\gamma(f),\gamma(g)\}$ hold whenever $f$ and $g$ are Baire class $1$ functions?
  \end{question}
 
  \begin{proposition}
    \label{beta_uniform_limit}
    If the sequence of Baire class 1 functions $f_n$ converges 
    uniformly to $f$ then $\beta(f) \le \sup_n \beta(f_n)$.
  \end{proposition}
  
  \begin{proof}
    If $|f- f_n| < \varepsilon / 3$ then $|\om(f, x, F) - \om(f_n, x, F)| \le \frac23 \eps$
    for every $x$ and $F$. Therefore $D_{f, \eps} (F) \subseteq D_{f_n, \eps/3} (F)$ for every $F$, which in turn implies
    $\beta(f, \varepsilon) \le \beta(f_n, \varepsilon / 3)$, from which
    the proposition easily follows.
  \end{proof}

  \begin{proposition}
    \label{gamma_uniform_limit}
    If the sequence of Baire class 1 functions $f_n$ converges 
    uniformly to $f$ then $\gamma(f) \lesssim \sup \limits_n \gamma(f_n)$. 
  \end{proposition}
  \begin{proof}
    By taking a subsequence we can suppose that 
    $|f_n(x) - f(x)| \le \frac{1}{2^n}$ for every $n \in \N$ and every $x \in X$. 
    With $g_n(x) = f_{n}(x) - f_{n - 1}(x)$ 
    we have $|g_n(x)| \le \frac{3}{2^n}$, hence $\sum_{n = 1}^\infty g_n(x)$ is uniformly convergent, 
    and $f(x) = f_0(x) + \sum_{n = 1}^\infty g_n(x)$. 
    Using Theorem \ref{beta_gamma_sum} we have 
    $\gamma(g_n) \lesssim \max\{\gamma(f_n), \gamma(f_{n - 1})\}$, hence 
    $\sup_n \gamma(g_n) \lesssim \sup_n \gamma(f_n)$. 
    It is enough to prove that for $g = \sum_{n = 1}^\infty g_n$ 
    we have $\gamma(g) \lesssim \sup_n \gamma(g_n)$, 
    since Theorem \ref{beta_gamma_sum} yields 
    $\gamma(f) \lesssim \max\{ \gamma(f_0), \gamma(g)\}$. 
    
    Now for every $n \in \N$ let $(\varphi^k_n)_{k \in \N}$ be a sequence of 
    continuous functions converging pointwise to $g_n$ with 
    $\sup_{\varepsilon > 0} \gamma((\varphi^k_n)_{k \in \N}, \varepsilon) = 
    \gamma(g_n)$. It is easy to see that we can suppose 
    $|\varphi^k_n(x)| \le \frac{3}{2^n}$ for every $n\in \N$ and $k \in \N$, 
    since by replacing $(\varphi^k_n)_{k \in \N}$ with 
    $\left(\max\left(\min\left(\varphi^k_n, \frac{3}{2^n}\right), 
    -\frac{3}{2^n}\right)\right)_{k \in \N}$ we have a 
    sequence of continuous functions satisfying this, and the sequence is still 
    converging pointwise to $g_n$, while 
    $\gamma((\varphi^k_n)_{k \in \N}, \varepsilon)$ is not increased. 
    
    Let $\phi_k = \sum_{n = 0}^k \varphi^k_n$. We show that $(\phi_k)_{k \in \N}$ 
    converges pointwise to $g$ and also that 
    $\gamma(g) \le \sup_{\varepsilon > 0} \gamma((\phi_k)_{k \in \N}, \varepsilon) 
    \lesssim \sup_n \sup_{\varepsilon > 0} \gamma((\varphi^k_n)_{k \in \N}, \varepsilon) =  
    \sup_n \gamma(g_n)$, which finishes the proof. To prove pointwise convergence, 
    let $\varepsilon > 0$ be arbitrary and fix $K \in \N$ with 
    $\frac{6}{2^{K}} < \varepsilon$. For $k > K$ we have 
    \begin{equation*}
      \left|\phi_k(x) - g(x)\right| = 
      \left|\sum_{n = 0}^k \varphi^k_n(x) - g(x) \right| \le
      \left|\sum_{n = 0}^K \varphi^k_n(x) - g(x) \right| + 
        \left|\sum_{n = K + 1}^k \varphi^k_n(x) \right|, 
    \end{equation*}
    where the first term of the last expression tends to 
    $\left| \sum_{n = 0}^K g_n(x) - g(x)\right| \le \frac{3}{2^{K}}$, 
    while the second is at most $\frac{3}{2^{K}}$. Hence 
    $\limsup_{k \to \infty} \left|\phi_k(x) - g(x)\right| \le 2\frac{3}{2^K} < \varepsilon$ 
    for every $\varepsilon > 0$, showing that $\phi_k(x) \to g(x)$. 
    
    Now fix an $\varepsilon > 0$ and $K \in \N$ as before, it is enough to show that 
    $\gamma((\phi_k)_{k \in \N}, 3\varepsilon) \lesssim 
    \sup_n \sup_{\varepsilon > 0} \gamma((\varphi^k_n)_{k \in \N}, \varepsilon)$. 
    
    For any $x \in X$ and $k, l > K$ we have 
    \begin{equation}
    \begin{split}
      \label{e:gamma_limit}
      \left|\phi_k(x) - \phi_l(x)\right| = 
      \left|\sum_{n = 0}^k \varphi^k_n(x) - \sum_{n = 0}^l \varphi^l_n(x) \right| \\
      \le \sum_{n = 0}^K\left|\varphi^k_n(x) - \varphi^l_n(x) \right| + 
        \left|\sum_{n = K + 1}^k \varphi^k_n(x) \right| + 
        \left|\sum_{n = K + 1}^l \varphi^l_n(x) \right|.
    \end{split}
    \end{equation}
    As before, the sum of the last two terms is at most $\varepsilon$. 
    We want to use Proposition \ref{ugly_lemma} for the derivatives 
    $D = D_{(\phi_k)_{k \in \N}, 3\varepsilon}$ and 
    $D_n = D_{(\varphi^k_n)_{k \in \N}, \frac{\varepsilon}{K + 1}}$ for $n \le K$. 
    To check condition \eqref{lemma1}, let 
    $x \in F \setminus \bigcup_{n = 0}^K D_{(\varphi^k_n)_{k \in \N}, \frac{\varepsilon}{K + 1}}(F)$. 
    Then we have a neighborhood $U$ of $x$ and an $N \in \N$ such that 
    $\left|\varphi^k_n(y) - \varphi^l_n(y)\right| < \frac{\varepsilon}{K + 1}$ 
    for every $n \le K$, every $y \in U \cap F$ and every $k, l \ge N$. This observation and 
    \eqref{e:gamma_limit} yields that $\left|\phi_k(y) - \phi_l(y)\right| \le 2\varepsilon$ for
    every $y \in U \cap F$ and $k,l \ge N$ showing that 
    $x \not \in D_{(\phi_k)_{k \in \N}, 3\varepsilon}(F)$. 
    
    Condition \eqref{lemma2} is similar, and it can be seen as in the 
    proof of Theorem \ref{beta_gamma_sum}. Now Proposition 
    \ref{ugly_lemma} gives 
    \begin{equation*}
      \gamma((\phi_k)_{k \in \N}, 3\varepsilon) \lesssim 
      \max_{n \le K} \gamma\left((\varphi^k_n)_{k \in \N}, \frac{\varepsilon}{K + 1}\right) \le 
      \sup_n \sup_{\varepsilon > 0} \gamma((\varphi^k_n)_{k \in \N}, \varepsilon), 
    \end{equation*}
    completing the proof. 
  \end{proof}

\begin{theorem}
  \label{a=b=g_bounded}
  If $f$ is a bounded Baire class 1 function then 
  $\alpha(f) \approx \beta(f) \approx \gamma(f)$. 
\end{theorem}
\begin{proof}
Using Theorem \ref{alpha_le_beta_le_gamma}, it is enough to prove that 
$\gamma(f) \lesssim \alpha(f)$. 
First, we prove the theorem for characteristic functions.
\begin{lemma}
  Suppose that $A \in {\boldsymbol\Delta^0_2}$. Then $\gamma(\chi_A) \lesssim \alpha(\chi_A)$. 
\end{lemma}
\begin{proof}
  In order to prove this, first we have to produce a sequence of continuous
functions converging pointwise to $\chi_A$. 

For this let $(F_\eta)_{\eta<\lambda}$ be a continuous transfinite
decreasing sequence of closed sets, so that \[A=\bigcup_{\begin{subarray}{c}
\eta < \lambda \\ \eta \text{ even} \end{subarray}} (F_\eta \setminus
F_{\eta+1})\] and $\lambda \approx \alpha(\chi_A)$ given by Corollary \ref{alpha_1_equals_alpha}. 
 We can assume that the last element of the sequence $(F_\eta)_{\eta<\lambda}$ is $\emptyset$, hence every
$x \in X$ is contained in a unique set of the form $F_\eta \setminus F_{\eta+1}$.

 For each $k \in \omega$ and $\eta<\lambda$ let $f^k_\eta:X \to [0,1]$ be a
continuous function so that $f^k_\eta | F_\eta \equiv 1$, and whenever $x \in X$
and $d(x,F_\eta) \geq \frac{1}{k + 1}$ then $f^k_\eta(x)=0$. Such a function exists
by Urysohn's lemma, since the sets $F_\eta$ and $\{x \in X: d(x,F_\eta)
\geq \frac{1}{k + 1}\}$ are disjoint closed sets.

Now let $(\eta_n)$ be an enumeration of $\lambda$ in type $\leq \omega$. Let us
define \[f_k=\sum_{\begin{subarray}{c}
n \leq k \\ \eta_n \text{ even} \end{subarray}} f^k_{\eta_n}-f^k_{\eta_n+1}.\]
 
Since the functions $f_k$ are finite sums of continuous functions, they are
continuous. We claim that $f_k \to \chi_A$ as $k \to \infty$. 

To see this, first let
$x \in X$ be arbitrary. Then there exists a unique $m$ so that $x \in F_{\eta_m}
\setminus F_{\eta_m+1}$. Choose $k \in \omega$ so that $k
\geq m$ and $d(x,F_{\eta_{m}+1}) \geq \frac{1}{k + 1}$. 

Then if $x \in A$ then
$\eta_m$ even and

\[f_k(x)=\sum_{\begin{subarray}{c}
n \leq k \\ \eta_n \text{ even} \end{subarray}}
f^k_{\eta_n}(x)-f^k_{\eta_n+1}(x)=\]
\[=\left(\sum_{\begin{subarray}{c}
n \leq k \\ \eta_n \text{ even} \\ \eta_n<\eta_m \end{subarray}}
f^k_{\eta_n}(x)-f^k_{\eta_n+1}(x)\right)+\left(\sum_{\begin{subarray}{c}
n \leq k \\ \eta_n \text{ even} \\ \eta_n>\eta_m  \end{subarray}}
f^k_{\eta_n}(x)-f^k_{\eta_n+1}(x)\right)+f^k_{\eta_m}(x)-f^k_{\eta_m+1}(x).\]

The first sum is clearly $0$ since $f^k_{\eta_n} \equiv 1$ on $F_{\eta_m}$ if
$\eta_m>\eta_n$. This is also true for the second one, since if
$d(x,F_{\eta_n}) \geq \frac{1}{k + 1}$ then $f^k_{\eta_n}(x)=0$. 
Finally, $f_{\eta_m}(x)=1$ and
$f_{\eta_m+1}(x)=0$, so $f_k(x)=1$. 

If $x \not \in A$ then $\eta_m$ is odd and

\[f_k(x)=\sum_{\begin{subarray}{c}
n \leq k \\ \eta_n \text{ even} \end{subarray}}
f^k_{\eta_n}(x)-f^k_{\eta_n+1}(x)=\]
\[=\sum_{\begin{subarray}{c}
n \leq k \\ \eta_n \text{ even} \\ \eta_n<\eta_m \end{subarray}}
f^k_{\eta_n}(x)-f^k_{\eta_n+1}(x)+\sum_{\begin{subarray}{c}
n \leq k \\ \eta_n \text{ even} \\ \eta_n>\eta_m  \end{subarray}}
f^k_{\eta_n}(x)-f^k_{\eta_n+1}(x).\]

Now the previous argument gives $f_k(x)=0$.

So $f_k \to \chi_A$ holds. Next we prove by induction on $\eta$ that for every
$\eta<\lambda$ and every $\varepsilon>0$ we have \[D^\eta_{(f_k)_{k \in 
\mathbb{N}},\varepsilon}(X) \subset F_\eta.\]
This will clearly complete the proof. 

For $\eta=0$ we have 
\[D^0_{(f_k)_{k \in 
\mathbb{N}},\varepsilon}(X) =X=F_0.\]

If $\eta$ is a limit ordinal, the statement is clear, since the sequence of derivatives as well as
$(F_\eta)_{\eta<\lambda}$ are continuous.

Now let $\eta=\theta+1$ and $D^\theta_{(f_k)_{k \in 
\mathbb{N}},\varepsilon}(X) \subset F_\theta$. For some $m$ we have
$\theta=\eta_m$. Let $x \in F_{\eta_m} \setminus F_{\eta_m+1}$. Then it is
enough to prove that $x \not \in D^\eta_{(f_k)_{k \in
\mathbb{N}},\varepsilon}(X).$ Let $k$ be so that $d(x,
F_{\eta_m+1})\geq \frac{2}{k + 1}$. 

If $d(x,y)<\frac{1}{k + 1}$ and $y \in
D^\theta_{(f_k)_{k \in
\mathbb{N}},\varepsilon}(X)$, then $y \in F_{\eta_m} \setminus F_{\eta_m+1}$.
From this, $l_1,l_2 \geq k$ implies that
$f^{l_1}_\eta(y)=f^{l_2}_\eta(y)=1$ if $\eta \leq \eta_{m}$, and
$f^{l_1}_\eta(y)=f^{l_2}_\eta(y)=0$ if $\eta>\eta_m$. Hence
$f_{l_1}(y)-f_{l_2}(y)=0.$ 

So the sequence $f_k$ is eventually
constant on a relative neighborhood of $x$ in $F_{\eta_m} $. 
Therefore $x \not \in D^\eta_{(f_k)_{k \in
\mathbb{N}},\varepsilon}(X)$, which finishes the proof. 
\end{proof}

Next we prove that $\gamma(f) \lesssim \alpha(f)$ for every step function $f$. We still need the 
    following lemma.
    \begin{lemma}
      If $A$ and $B$ are ambiguous sets then 
      $$\alpha\left(\chi_{A \cap B}\right) \lesssim 
        \max\left\{\alpha\left(\chi_A\right), 
        \alpha\left(\chi_B\right) \right\}.
      $$
    \end{lemma}
    \begin{proof}
      It is enough to prove this for $\beta$ since the previous 
      lemma and Theorem \ref{alpha_le_beta_le_gamma} yields that the 
      ranks essentially agree on characteristic functions. Theorem 
      \ref{beta_gamma_sum} gives 
      $\beta(\chi_A + \chi_B) \lesssim 
      \max\{ \beta(\chi_A), \beta(\chi_B)\}$, hence it suffices to prove that $\beta \left(\chi_{A \cap B}\right) \le \beta(\chi_A + \chi_B)$.
      But this easily follows, since one can readily check that
      for every $\varepsilon < 1$  and $F$ we have 
      $D_{\chi_{A \cap B}, \varepsilon}(F) \subseteq 
      D_{\chi_A + \chi_B, \varepsilon}(F)$, finishing the proof.      
    \end{proof}
    
    Now let $f$ be a step function, so 
    $f = \sum_{i = 1}^{n} c_i \chi_{A_i}$, where the $A_i$'s are
     disjoint ambiguous sets covering $X$, and 
    we can also suppose that the $c_i$'s form a strictly increasing 
    sequence of real numbers. 
    \begin{lemma}
      \label{small_lemma}
      $\max_i \{\alpha(\chi_{A_i})\} \lesssim \alpha(f)$.
    \end{lemma}
    \begin{proof}
      Let $H_i = \bigcup_{j = 1}^i A_j$. By the definition of 
      the rank $\alpha$, for every $i$ we have
      \begin{equation}
        \label{alpha_something}
        \alpha(H_i, H_i^c) \le \alpha(f).
      \end{equation} 
      This shows that 
      $\alpha(\chi_{A_1}) \lesssim \alpha(f)$, and together with the 
      previous lemma, for $i > 1$ 
      \begin{equation*}
      \begin{split}
        \alpha(\chi_{A_i}) &= \alpha(\chi_{H_i \setminus H_{i - 1}}) = 
          \alpha(\chi_{H_i \cap H^c_{i - 1}}) \lesssim
          \max\{\alpha(\chi_{H_i}), \alpha(\chi_{H^c_{i - 1}})\} \\ 
        &= \max\{\alpha(H_i, H^c_i), \alpha(H_{i - 1}, H^c_{i - 1}) 
          \} \le \alpha(f),
      \end{split} 
      \end{equation*}
      where the last but one inequality follows from the above lemma and
      the last inequality from \eqref{alpha_something}.
    \end{proof}
    
    Now we have 
    \begin{equation*}
      \gamma(f) \lesssim \max_i \{\gamma(\chi_{A_i})\} 
      \approx \max_i \{ \alpha(\chi_{A_i})\} \lesssim \alpha(f), 
    \end{equation*}
    where we used Theorem \ref{beta_gamma_sum}, this theorem for 
    characteristic functions and 
    Lemma \ref{small_lemma}, proving the theorem for step functions.
   
    In particular, $\alpha(f) \leq  \beta(f) \leq \gamma(f)$ (Theorem \ref{alpha_le_beta_le_gamma}) gives the following corollary.
   
   \begin{corollary}
    \label{c:r_stepf} If $f=\sum_{i = 1}^{n} c_i \chi_{A_i}$, where the $A_i$'s are
     disjoint ambiguous sets covering $X$ and the $c_i$'s are distinct then \[\alpha(f) \approx\max_i \{\alpha(\chi_{A_i})\}\] and similarly for $\beta$ and $\gamma$.
   \end{corollary}

    Now let $f$ be an arbitrary bounded Baire class 1 function.
    \begin{lemma}
      \label{f_n_uniformly_conv}
      There is a sequence $f_n$ of step functions converging 
      uniformly to $f$, satisfying 
      $\sup_n \alpha(f_n) \lesssim \alpha(f)$.
    \end{lemma}

    \begin{proof}%[Proof of Lemma \ref{f_n_uniformly_conv}]
      Let $p_{n, k} = k / 2^n$ for all $k \in \Z$ and $n \in \N$. The 
      level sets $\{f \le p_{n, k}\}$ and $\{f \ge p_{n, k + 1}\}$ are 
      disjoint ${\boldsymbol \Pi^0_2}$ sets, hence they can be separated 
      by a $H_{n, k} \in {\boldsymbol \Delta^0_2}(X)$ 
      (see e.g. \cite[22.16]{K}). We can choose $H_{n, k}$ to 
      satisfy $\alpha_1(H_{n, k}, H^c_{n, k}) \le 2\alpha(f)$ using 
      Proposition \ref{alpha_1_equals_alpha_for_sets}. 
      
      Since $f$ is bounded, for fixed $n$ there are only finitely 
      many $k \in \Z$ for which 
      $H_{n, k + 1} \setminus H_{n, k} \neq \emptyset$. Set 
      \begin{equation*}
        f_n = \sum_{k \in \Z} p_{n, k} \cdot 
          \chi_{H_{n, k + 1} \setminus H_{n, k}}.
      \end{equation*}
      Now for each $n$, $f_n$ is a step function with 
      $|f - f_n| \le 2^{n - 1}$. Hence $f_n \to f$ uniformly. Since 
      the level sets of a function $f_n$ are of the form $H_{n, k}$ or 
      $H^c_{n, k}$ for some $k \in \Z$, we have 
      $\alpha(f_n) \le 2\alpha(f)$, proving the lemma. 
    \end{proof}
    
    Let $f_n$ be a sequence of step functions given by this lemma. 
    Using Proposition \ref{gamma_uniform_limit} and this theorem for 
    step functions, we have 
    $\gamma(f) \lesssim \sup_n\gamma(f_n) \lesssim \sup_n\alpha(f_n) 
    \lesssim \alpha(f)$, completing the proof.
\end{proof}

We have seen above that $\alpha$ is not essentially additive on the Baire class $1$ functions but $\beta$ and $\gamma$ are, therefore $\alpha$ cannot essentially coincide with $\beta$ or $\gamma$. However, in view of the above theorem the following question arises. 

  \begin{question}
  \lab{q:b=g}
   Does $\beta \approx \gamma$ hold for arbitrary Baire class $1$ functions? 
  \end{question}

%  \begin{question}
%    Does $\beta \approx \gamma$ hold for \emph{bounded} Baire class $1$ functions?
%  \end{question}

%Actually, we do not even know of an unbounded counterexample.

  \begin{proposition}
    \label{alpha_uniform_limit}
    If the sequence of Baire class 1 functions $f_n$ converges 
    uniformly to $f$ then $\alpha(f) \lesssim \sup\limits_n 
    \alpha(f_n)$.
  \end{proposition}
  \begin{proof}
    If $f$ is bounded (hence without loss of generality the $f_n$ are 
    also bounded) this is an easy consequence of 
     Theorem \ref{a=b=g_bounded} and Proposition 
    \ref{beta_uniform_limit}. 
    
    For an arbitrary function $g$ let $g' = \arctan \circ g$. 
    It is easy to show that $\alpha(g') = \alpha(g)$ using 
    Remark \ref{alpha_def_real}. 
    
    If the functions $f$ and $f_n$ are given such that $f_n \to f$ 
    uniformly then $f'_n \to f'$ uniformly, and these are bounded 
    functions, so we have 
    $\alpha(f) = \alpha(f') \lesssim \sup\limits_n \alpha(f'_n) = 
    \sup\limits_n \alpha(f_n)$. 
  \end{proof}

%  \begin{question}
%  Suppose that $f_n$ is a sequence of Baire class $1$ functions converging
%    uniformly to $f$. Does $\gamma(f) \lesssim \sup\limits_n 
%    \gamma(f_n)$ hold?
%  \end{question}

\section{Ranks on the Baire class $\xi$ functions exhibiting strange phenomena}
\lab{s:neg}

\subsection{The separation rank and the linearized separation rank}
  The only rank out of the ones discussed above that has straightforward generalization to the Baire class $\xi$ case is the 
  rank $\alpha_1$. However, this generalization does not answer Question \ref{q:EL}, since, similarly to the original $\al_1$, it is not linear. 
  After discussing this, we will propose a very natural modification that transforms an arbitrary rank into a linear one, but we will see that this modified rank is bounded in $\om_1$ for characteristic functions! 
  
  \begin{definition}
    Let $A$ and $B$ be disjoint ${\boldsymbol \Pi^0_{\xi + 1}}$ sets. Then they can be 
    separated by a ${\boldsymbol \Delta^0_{\xi + 1}}$ set 
    (see e.g. \cite[22.16]{K}). Since every ${\boldsymbol \Delta^0_{\xi + 1}}$ set is the transfinite 
    difference of ${\boldsymbol \Pi^0_\xi}$ sets, $A$ and $B$ can be separated
    by the transfinite difference of such a sequence. Let $\alpha_{\xi}(A,B)$ denote the length of the shortest such 
    sequence. 
  \end{definition}

  \begin{definition}
    Let $f$ be a Baire class $\xi$ function, and $p < q \in \Q$. Then $\{f \le p\}$ and $\{f \ge q\}$ are disjoint ${\boldsymbol \Pi^0_{\xi + 1}}$ sets. Let the \emph{separation rank} of $f$ be 
    $$
      \alpha_\xi(f) = \sup_{\begin{subarray}{c} p < q \\ p, q \in \Q 
      \end{subarray}} \alpha_\xi(\{f \le p\}, \{f \ge q\}).
    $$
  \end{definition}
  Note that this really extends the definition of $\alpha_1$.
  
  \begin{theorem}
    \label{alpha_not_bounded}
    For every $1 \le \xi < \omega_1$ the rank $\alpha_\xi$ is unbounded in $\omega_1$ on the characteristic Baire class 
    $\xi$ functions.
  \end{theorem}
  \begin{proof}
  Let 
    $\mathcal{U} \in {\boldsymbol \Pi^0_\xi}(2^\omega \times X)$ 
    be a universal set for ${\boldsymbol \Pi^0_\xi}(X)$ sets, that is, 
    for every $F \subseteq X$, $F \in {\boldsymbol \Pi^0_\xi}(X)$ 
    there exists a $y \in 2^\omega$ such that $\mathcal{U}^y = F$. 
    For the existence of 
    such a set see \cite[22.3]{K}. Let us use the notation $\Gamma_\zeta(X)$ 
    for the the family of sets 
    $H \subseteq X$ satisfying $\alpha_\xi (H, H^c) < \zeta$. 
    From \cite[22.27]{K} we have 
    $\Gamma_\zeta(X) \subseteq {\boldsymbol \Delta^0_{\xi + 1}(X)}$. 
    We will show that there exists a 
    ${\boldsymbol \Delta^0_{\xi + 1}}$ set for every 
    $\zeta < \omega_1$ which is universal for the family of 
    $\Gamma_\zeta$ sets. Since $X$ is uncountable, there is a 
    continuous embedding of $2^\omega$ into $X$ (\cite[6.5]{K}), hence 
    no universal 
    set exists in $2^\omega \times X$ for the family of 
    ${\boldsymbol \Delta^0_{\xi + 1}}(X)$ 
    sets (easy corollary of \cite[22.7]{K}). 
    This implies for every $\zeta < \omega_1$ that 
    $\Gamma_\zeta \neq {\boldsymbol \Delta^0_{\xi + 1}}$, hence the rank 
    is really unbounded.

    Let $p:\zeta \times \mathbb{N} \to \mathbb{N}$ be a bijection.
    For $\eta < \zeta$ and $y \in 2^\omega$ we define 
    $\phi(y)_\eta \in 2^\omega$ by  
    $\phi(y)_\eta(n) = y(p(\eta, n))$. 
    First we check that for a fixed $\eta < \zeta$ the map 
    $y \mapsto \phi(y)_\eta$ is continuous. Let 
    $U = \{ x \in 2^\omega : x(0) = i_0, \dots, x(n) = i_n\}$ be a 
    set from the usual basis of $2^\omega$. The preimage of $U$ is the set 
    $\{y \in 2^\omega : \forall k \le n\; \phi(y)_\eta(k) =i_k\} = 
    \{y \in 2^\omega : \forall k \le n\; y(p(\eta, k)) = i_k\}$, which is a basic open set, too. Now 
    $\mathcal{U}_\eta = \{(y, x) : (\phi(y)_\eta, x) \in \mathcal{U}\}$ 
    is a continuous preimage of a ${\boldsymbol \Pi^0_\xi}$ 
    set, hence 
    $\mathcal{U}_\eta \in {\boldsymbol \Pi^0_\xi}(2^\omega \times X)$ 
    (see \cite[22.1]{K}). 
    Let  
    \begin{equation*}
    \begin{split}
      \mathcal{U}' = \{(y, x) \in 2^\omega \times X : 
        \text{the smallest ordinal $\eta$ such that 
        $(y, x) \not \in \mathcal{U}_\eta$ is odd,} \\ 
      \text{if such an $\eta$ exists, or no such $\eta$ exists and 
      $\zeta$ is odd} \}.
    \end{split}
    \end{equation*}
    Now we check that $\mathcal{U}' \in {\boldsymbol \Delta^0_{\xi + 1}}
    (2^\omega \times X)$. Let $\iV_\eta = \bigcap_{\theta < \eta} \iU_\theta$, then
    these sets form a continuous decreasing sequence of ${\boldsymbol \Pi^0_\xi}$ sets and it is easy to see that ${\iU'}^c$ 
    is the transfinite difference of the sequence 
    $(\iV_\eta)_{\eta < \zeta + 1}$, hence ${\iU'}^c \in {\boldsymbol \Delta^0_{\xi+1}}$, proving that 
    $\iU' \in {\boldsymbol \Delta^0_{\xi+1}}$, since the family 
    of ${\boldsymbol \Delta^0_{\xi + 1}}$ sets is closed under 
    complements (see \cite[22.1]{K}).
    
    Now we show that $\iU'$ is universal. For a set $H \in \Gamma_\zeta(X)$ there is a sequence 
    $(z_\eta)_{\eta < \zeta}$ in $2^\omega$, such that $H$ is the 
    transfinite difference of the sets $\mathcal{U}^{z_\eta}$. 
    For every sequence $(z_\eta)_{\eta < \zeta}$ we can find  
    $y \in 2^\omega$ such that $\phi(y)_\eta = z_\eta$. Namely 
    $y : p(\eta, n) \mapsto z_\eta(n)$ makes sense (since $p$ is a 
    bijection), and works. Consequently, for $H$ there is  
    $y \in 2^\omega$, such that $H$ is the transfinite 
    difference of the sets 
    $\iU^{z_\eta} = \mathcal{U}^{\phi(y)_\eta} = \left(\mathcal{U}_\eta\right)^y$. It is easy to see that if $H$ 
    is the transfinite difference of the sequence 
    $\left(\left(\mathcal{U}_\eta\right)^y\right)_{\eta < \zeta}$ 
    then 
    \begin{equation*}
    \begin{split}
      H = \{x \in X : \text{the smallest ordinal $\eta$ such 
    that $x \not \in \left(\mathcal{U}_\eta\right)^y$ is odd,} \\ 
    \text{if such an $\eta$ exists, or no such $\eta$ exists and 
      $\zeta$ is odd}\},      
    \end{split}
    \end{equation*}
    hence $H = {\mathcal{U}'}^y$.
  \end{proof}
  
  \begin{corollary}
    \label{alpha_unbounded_on_perfect} For every $1 \le \xi < \omega_1$, every non-empty perfect set 
    $P \subseteq X$ and every ordinal $\zeta < \omega_1$ there is 
    a characteristic function $\chi_A \in \mathcal{B}_\xi(X)$ 
    with $A \subseteq P$ and $\alpha_\xi(\chi_A) \ge \zeta$. 
  \end{corollary}
  \begin{proof}
    Since $P$ is perfect, it is an uncountable Polish space with the 
    subspace topology, hence the rank $\alpha_\xi$ is unbounded 
    on the characteristic Baire class $\xi$ functions defined on $P$ by the
    previous theorem. 
    Hence we can take a characteristic function 
    $f' \in \mathcal{B}_\xi(P)$ with $\alpha_\xi(f') \ge \zeta$, and 
    set 
    \begin{equation*}
      f(x) = \left\{ \begin{array}{cl}
        f'(x) & \text{if $x \in P$} \\
        0 & \text{if $x \in X \setminus P$}.     
      \end{array} \right.
    \end{equation*}
    It is easy to see that $f \in \mathcal{B}_\xi(X)$, hence 
    it is enough to prove that $\alpha_\xi(f) \ge \zeta$. 
    
    For this, it is enough to prove that 
    $\alpha_\xi(\{f' \le p\}, \{f' \ge q\}) \le 
    \alpha_\xi(\{f \le p\}, \{f \ge q\})$ for every pair of rational 
    numbers $p < q$. For this, let 
    $H \in {\boldsymbol \Delta^0_{\xi + 1}}(X)$ where 
    $\{f \le p\} \subseteq H \subseteq \{f \ge q\}^c$ and 
    $H$ is the transfinite difference of the sets 
    $(F_{\eta})_{\eta < \lambda}$ with 
    $\lambda = \alpha_\xi(\{f \le p\}, \{f \ge q\})$ and 
    $F_{\eta} \in {\boldsymbol \Pi^0_\xi}(X)$ for every 
    $\eta < \lambda$. 
    
    Let $H' = P \cap H$ and for every 
    $\eta < \lambda$ let $F'_{\eta} = P \cap F_{\eta}$. 
    It is easy to see that $H'$ separates the level sets 
    $\{f' \le p\}$ and $\{f' \ge q\}$ and $H'$ is the transfinite 
    difference of the sets $(F'_{\eta})_{\eta < \lambda}$. 
    And since $H' \in {\boldsymbol \Delta^0_{\xi + 1}}(P)$ and 
    $F'_{\eta} \in {\boldsymbol \Pi^0_\xi}(P)$  for every 
    $\eta < \lambda$ 
    (\cite[22.A]{K}), we have the desired inequality 
    $\alpha_\xi(\{f' \le p\}, \{f' \ge q\}) \le 
    \alpha_\xi(\{f \le p\}, \{f \ge q\})$. Thus the proof is complete. 
  \end{proof}
  
  The main disadvantage of this rank is that the construction of Remark \ref{alpha_not_sum} easily yields that
  the rank does not behave nicely under linear operations. We leave the easy proof of the next statement to the reader. 
  
  \begin{proposition}
  Let $1 \le \xi < \om_1$. Then $\al_\xi$ is not essentially linear, actually not even essentially additive. 
  \end{proposition}
    
  However, there is a natural way to make a rank linear. 
  
  \begin{definition}
    For an $f \in \mathcal{B}_\xi$, let 
    \begin{align*}
      \alpha'_\xi (f) = \min \{ \max \{ \alpha_\xi (f_1), 
        \dots, \alpha_\xi (f_n) \} : 
      &n \in \omega, f_1, \dots, f_n \in  \mathcal{B}_\xi, \\ 
      &f = f_1+\dots+f_n \}. 
    \end{align*}
  \end{definition}
  
  It can be easily seen that $\alpha'_\xi$ is now linear, 
  but we do not know whether it is still unbounded in $\omega_1$. 
  
  \begin{question}
  \lab{q:a'}
  Let $1 \le \xi < \om_1$. Is $\alpha'_\xi$ unbounded in $\omega_1$?
  \end{question}
    
    We have the following partial result, which is a very strong indication that the answer to this question is in the negative, since in every single case when we can show that a rank is unbounded it is actually unbounded on the characteristic functions.
    
  \begin{theorem}
  \lab{t:a'}
   If $1 \le \xi < \om_1$ and $f$ is a \emph{characteristic} Baire class $\xi$ function then $\alpha'_\xi(f) \le 2$.
  \end{theorem}
  
  \begin{proof}
    Let us call a function $f$ a semi-Borel class $\xi$ function if 
    the level sets $\{f < c\}$ are in ${\boldsymbol \Sigma^0_\xi}$ for every $c \in \R$.
    Note that then the level sets $\{f > c\}$ are in 
    ${\boldsymbol \Sigma^0_{\xi + 1}}$, hence $f \in \mathcal{B}_\xi$. 
    
    We first show that a semi-Borel class $\xi$ function has 
    $\alpha_\xi$ rank at most $2$. Let $p < q$ be a pair of rational 
    numbers. The level set 
    $\{f \ge q\} \in {\boldsymbol \Pi^0_\xi}(X)$, hence the 
    transfinite difference of the 
    sequence $F_0 = X, F_1 = \{f \ge q\}$ separates the level sets 
    $\{f \le p\}$ and $\{f \ge q\}$. 
    
    Now using the same idea as in Remark \ref{alpha_not_sum}, it is 
    clear that every characteristic Baire class $\xi$ function can be written as the 
    difference of two semi-Borel class $\xi$ functions, completing the proof 
    of this theorem. 
  \end{proof}

The following question is very closely related to Question \ref{q:a'}.

 \begin{question}
    Let $1 \le \xi < \om_1$ and let $f_n$ and $f$ be Baire class $\xi$ functions such that $f_n \to f$ uniformly. Does this imply that $\alpha'_\xi (f) \lesssim \sup_n \alpha'_\xi (f_n)$?
  \end{question}
  
\begin{remark}
An affirmative answer to this question would provide a negative answer to Question \ref{q:a'}.
Indeed, it is not hard to show that $\alpha'_\xi$ is 
  bounded for step functions, and hence, by taking uniform 
  limit, for every bounded function. Then one can check that 
  the rank of an arbitrary function $f$ equals to the rank of the 
  bounded function $\arctan \circ f$, hence $\alpha'_\xi$ 
  is bounded. 
\end{remark}

\subsection{Limit ranks}

  In this section we apply an even more natural approach to define ranks on the Baire class $\xi$ functions 
  starting from an arbitrary rank on the Baire class 1 functions. Surprisingly, they will all turn out to be bounded in $\om_1$.
  
  \begin{definition}
    Let $\rho$ be a rank on the Baire class 1 functions. We 
    inductively define a rank $\overline{\rho}_\xi$ on the Baire class $\xi$ 
    functions. First, let $\overline{\rho}_1 = \rho$. For a successor ordinal $\xi + 1$ and a Baire class $\xi+1$ function $f$ let 
    \begin{equation*}
      \overline{\rho}_{\xi + 1}(f) = 
        \min \left\{ \sup_n \overline{\rho}_\xi(f_n) : 
        f_n \to f,\; \text{$f_n$ is of Baire class $\xi$} \right\}.
    \end{equation*}
Finally, for a limit ordinal $\xi$ and a Baire class $\xi$ function $f$ let
    \begin{equation*}
    \begin{split}
      \overline{\rho}_\xi (f) =
        \min \bigg\{ &\sup_n \overline{\rho}_{\xi_n} (f_n) : f_n \to f,\; 
        \text{$f_n$ is of Baire class $\xi_n$, $\xi_n < \xi$,} \\ 
        & \text{$f_n$ is not of Baire class $\zeta$ if  
        $\zeta < \xi_n$} \bigg\}.
    \end{split}
    \end{equation*}
  \end{definition}

 Surprisingly, the ranks $\overline{\alpha}_\xi$, $\overline{\beta}_\xi$ and 
 $\overline{\gamma}_\xi$ will all be bounded for $\xi \ge 2$.
 
  \begin{theorem}
  \lab{t:limit_bdd}
    If $2 \le \xi <\om_1$ then 
    $\overline{\alpha}_\xi \le \overline{\beta}_\xi \le \overline{\gamma}_\xi \le \omega$.
  \end{theorem}
  
  \begin{proof}
    It is enough to prove the theorem for $\xi = 2$. 
    Let $\Phi$ be a class of real valued functions on $X$. As in \cite{hausdorff}, we say 
    that $\Phi$ is \textit{ordinary} if it contains the constant 
    functions and if $f, g \in \Phi$ then $\max(f, g)$, 
    $\min(f, g)$, $f + g$, $f - g$, $fg$ and $f / g$ (if $g$ is 
    nowhere zero) are all in $\Phi$. 
    An ordinary class of functions is called \textit{complete} if it 
    is closed under uniform limits.
    
    For a class of functions $\Phi$, we denote by $\Phi^p$ the set 
    of functions that are pointwise limits of functions from $\Phi$. 
    We denote the pair of families of level sets of functions in 
    $\Phi$ by $\mathcal{P}(\Phi)$, that is, 
    $$
      \mathcal{P}(\Phi) = \left( 
        \left\{ \{f > c\} : f \in \Phi, c \in \mathbb{R} \right\}, 
        \left\{ \{f \ge c\} : f \in \Phi, c \in \mathbb{R} \right\}
      \right).
    $$
    
    If $\mathcal{P} = (\mathcal{M}, \mathcal{N})$ is a pair of 
    systems of sets then we denote the class of 
    functions whose levels sets are in $\mathcal{P}$ by 
    $\Phi(\mathcal{P})$, that is, 
    $$
      \Phi(\mathcal{P}) = \left\{ f : X \to \R \ | \  \forall c \in \mathbb{R} \; 
        \{f > c\} \in \mathcal{M}, \{f \ge c\} \in \mathcal{N} 
        \right\}.
    $$

    Now we state three theorems based on results in 
    \cite{hausdorff}.
    
    \begin{theorem}
      \label{ordinary_limit_complete}
      If a class of functions $\Phi$ is ordinary then 
      $\Phi^p$ is ordinary and complete.
    \end{theorem}
    
    \begin{theorem}
      \label{ordinary_pointwise_limit}
      If a class of functions $\Phi$ is ordinary and 
      $\mathcal{P}(\Phi) = (\mathcal{M}, \mathcal{N})$  then 
      $\mathcal{P}(\Phi^p) = 
      (\mathcal{N}_{\delta \sigma}, \mathcal{M}_{\sigma \delta})$.
    \end{theorem}
    
    \begin{theorem}
      \label{complete_ordinary_level_sets}
      If a class of functions $\Phi$ is complete and ordinary then 
      $\Phi = \Phi(\mathcal{P}(\Phi))$.
    \end{theorem}

    Theorem \ref{ordinary_limit_complete} is shown in 
    \cite[\textsection 41.~IV.]{hausdorff}, Theorem 
    \ref{ordinary_pointwise_limit} is an easy corollary of 
    \cite[\textsection 41.~V., VI.]{hausdorff} and 
    Theorem \ref{complete_ordinary_level_sets} is shown in 
    \cite[\textsection 41.~VIII.]{hausdorff}.
    
    Now let $\Phi$ consist of the Baire class 1 functions of the form 
    $$
      \sum_{i = 1}^n c_i \chi_{H_i},
    $$
    where $H_i$ is in the algebra $\mathcal{A}$ generated by the open 
    sets (an algebra is a family closed under finite unions and complements). It is easy to check that $\mathcal{A}$ contains exactly 
    the sets that can be written as the finite disjoint union of sets of 
    the form $F \cap G$, where $F$ is closed and $G$ is open. 
    Indeed, the intersection of two such set is of the same form, 
    and the complement of such a set is 
    \begin{equation*}
    \begin{split}
      \left(\bigcup_{i=0}^{n-1} (F_i \cap G_i)\right)^c = 
        \bigcap_{i = 0}^{n - 1} \left(F_i \cap G_i \right)^c = 
        \bigcap_{i = 0}^{n - 1} \left({F^c_i} \cup {G^c_i}\right) = \\
      \bigcup \left\{ 
        \bigcap_{i = 0}^{n - 1} F^{a(i)}_i \cap 
        \bigcap_{i = 0}^{n - 1} G^{b(i)}_i : 
        \text{$a, b \in 2^n$, 
         $\forall i<n$ at least one of $a(i)$ and $b(i)$ is 1} \right\}, 
    \end{split}
    \end{equation*}
    where for a set $H$, $H^0 = H$ and $H^1 = H^c$, and the last 
    equality holds, since a point $x$ is contained in either of the two
    sets in question iff for every $i < n$ it is contained in at least one 
    of $F^c_i$ and $G^c_i$. Now we check that the sets in the union are disjoint. 
    Without loss of generality we have two terms with distinct $a$'s, so $a(i) = 0$ and $a'(i) = 1$ for a suitable $i$.
   But then the term belonging to $a$ is a subset of ${F_i}$ and the other one is a subset of $F^c_i$, proving disjointness.
    
    An easy consequence of these observations is that 
    $\Phi$ is ordinary. 
    
    \begin{lemma}
      $\gamma(f) \le \omega$ for every $f \in \Phi$.
    \end{lemma}
    \begin{proof}
      First we prove that $\gamma(\chi_F) \le 2$ for every closed set $F$. Let $F$ be a closed set, and define 
      $f_n(x) = 1 - \min\{1, n \cdot d(x, F)\}$. It is easy to check 
      that $f_n \to \chi_F$ pointwise. 
      We now show that $\gamma((f_n)_{n \in \N}, \eps) \le 2$ for every $\eps > 0$, which will imply $\gamma(\chi_F) \le 2$. 
      Fix $\eps > 0$. If $x \not\in F$ then $x$ has a neighborhood $U$ such that $d(U, F) > 0$ and then if we fix an $N > \frac{1}{d(U, F)}$ then $f_n(y) = 0$ for every $y \in U$ and $n \ge N$, therefore $\omega((f_n)_{n \in \N}, x, X) = 0$. This implies $D_{(f_n)_{n \in \N}, \varepsilon}(X) \subseteq F$. But 
      $f_n|_F \equiv 1$ for every $n$, hence if $x \in F$ then 
      $\omega((f_n)_{n \in \N}, x, F) = 0$, therefore
      $D^2_{(f_n)_{n \in \N}, \varepsilon}(X) \subseteq D_{(f_n)_{n \in \N}, \varepsilon}(F) = \emptyset$, proving $\gamma((f_n)_{n \in \N}, \eps) \le 2$.
      
      It is easy to check that $\gamma(f) = \gamma(1-f)$ for every $f \in \iB_1$. This implies that $\gamma(\chi_G) \le 2$ for every open set $G$, since $\chi_G = 1 - \chi_{X \setminus G}$.
      
      Now, let $H = F \cap G$, where $F$ is closed and $G$ is open. 
      We show that $\gamma(\chi_H) \le \omega$.
      By Theorem \ref{beta_gamma_sum} there exists
      a sequence $f_n$ of continuous functions with $f_n \to \chi_F + \chi_G$ and 
      $\gamma((f_n)_{n \in \N}, \eps) \le \omega$ for every $\eps>0$. Define $f_n' = \max\{0, f_n - 1\}$. Then it is 
      easy to check that $f_n' \to \chi_H$ and 
      $\gamma((f_n')_{n \in \N}, \eps) \le \gamma((f_n)_{n \in \N}, \eps) \le \omega$ for every $\eps>0$.
      
      Since any $H \in \mathcal{A}$ is a finite disjoint union of sets 
      of the form $F \cap G$, the above paragraph shows that 
      $\chi_H = \chi_{H_0} + \dots + \chi_{H_n}$, where 
      $\gamma(\chi_{H_i}) \le \omega$. But then Theorem 
      \ref{beta_gamma_sum} yields that $\gamma(\chi_H) \le \omega$. 
      Then applying Theorem \ref{beta_gamma_sum} once again we obtain 
      that $\gamma(f) \le \omega$ for every $f \in \Phi$.
    \end{proof}
    
    Now we turn to the proof of the theorem. By Theorem 
    \ref{alpha_le_beta_le_gamma} and the previous lemma, it is enough 
    to show that $\Phi^p$ equals the family of Baire class 2 
    functions. Since every $f \in \Phi$ is of Baire class 1, we have 
    that $\Phi^p$ is a subclass of the Baire class 2 functions.
    
    For the converse, let us define $\mathcal{M}$ and $\mathcal{N}$ by 
    $\mathcal{P}(\Phi) = (\mathcal{M}, \mathcal{N})$. 
    By the definition of $\Phi$, $\mathcal{M}$ and $\mathcal{N}$ both 
    contain the open and closed sets. By Theorem 
    \ref{ordinary_pointwise_limit} $\mathcal{P}(\Phi^p) = 
    (\mathcal{N}_{\delta \sigma}, \mathcal{M}_{\sigma \delta})$, 
    hence 
    ${\boldsymbol \Sigma^0_3} \subseteq \mathcal{N}_{\delta \sigma}$ 
    and ${\boldsymbol \Pi^0_3} \subseteq \mathcal{M}_{\sigma \delta}$. 
    And by Theorem \ref{ordinary_limit_complete} and Theorem
    \ref{complete_ordinary_level_sets}
    $
      \Phi^p = \Phi(\mathcal{P}(\Phi^p)) = 
      \Phi(\mathcal{N}_{\delta \sigma}, \mathcal{M}_{\sigma \delta})
      \supseteq \Phi({\boldsymbol\Sigma^0_3}, {\boldsymbol\Pi^0_3}) = \mathcal{B}_2,
    $
    finishing the proof.
  \end{proof}

\subsection{Partition ranks}

  The following well known fact also gives rise to a very natural rank on the Baire class $\xi$ functions. However, this also turns out to be bounded. 
  
  \begin{proposition}
    A function $f$ is of Baire class $\xi$ if and only if for every 
    $\varepsilon > 0$ there exists a function $g$ of the form 
    $g = \sum_{n \in \omega} c_n \cdot \chi_{H_n}$, where 
    $H_n \in {\boldsymbol \Delta^0_{\xi + 1}}(X)$, the $H_n$'s form 
    a partition of $X$ and $|f(x) - g(x)| \le \varepsilon$ for every 
    $x \in X$. Moreover, if $f$ is bounded then each set $H_n$ can be 
    chosen to be empty for all but finitely many $n \in \omega$.
  \end{proposition}
  \begin{proof}
    If $f$ is of Baire class $\xi$ then for a fixed 
    $\varepsilon > 0$ let the numbers $p_n$ be defined by 
    $p_{n} = n \cdot \frac{\varepsilon}{2}$ for 
    every $n \in \mathbb{Z}$. The sets 
    $\{f \le p_n\}$ and $\{f \ge p_{n + 1}\}$ are disjoint 
    ${\boldsymbol \Pi^0_{\xi +1}}$ sets, hence they can be separated 
    by a set $A_n \in {\boldsymbol \Delta^0_{\xi + 1}}$. Now let 
    $H_n = A_n \setminus A_{n - 1}$. Note that if $f$ is bounded then 
    $H_n = \emptyset$ for all but finitely many $n \in \omega$.
    These sets form a partition, and
    with $g = \sum_{n \in \mathbb{Z}} p_n \cdot \chi_{H_n}$ the proof 
    of the first direction is complete. 
    
    For the other one, note that the function $g$ is of Baire class 
    $\xi$, hence $f$ is the uniform limit of Baire class $\xi$ 
    functions, implying that $f$ is of Baire class $\xi$ (see e.g. 
    \cite[24.4]{K}).
  \end{proof}
  \begin{definition}
    Let $f$ be a Baire class $\xi$ function and let the 
    \emph{partition rank} of $f$ be
    \begin{equation*}
    \begin{split}
      \delta(f) = \sup_{\varepsilon > 0}  \min\bigg\{ 
        \sup_{n \in \omega} \alpha_\xi(H_n, H^c_n) : 
        \; H_n \in {\boldsymbol \Delta^0_{\xi + 1}}, \; 
          \bigcup_{n \in \om} H_n = X,\\
        H_n \cap H_m = \emptyset \;(n \neq m), \;
          \exists (c_n)_{n \in \omega} \; 
          \bigg| f - \sum_{n\in \omega} c_n \cdot \chi_{H_n} \bigg| 
          \le  \varepsilon \bigg\}.
    \end{split}
    \end{equation*}
  \end{definition}
  
  \begin{proposition}
    $\delta(f) \le 4$ for every Baire class $\xi$ function $f$. 
  \end{proposition}
  \begin{proof}
  Fix $\eps > 0$. Obtain a function of the form $\sum_{n \in \omega} c_n \cdot \chi_{H_n}$ as in the above proposition.    
    It is enough to prove that every $H_n$ has a further partition into a sequence of sets $H_{n,k} \in {\boldsymbol \Delta^0_{\xi + 1}}$ with $\alpha_\xi(H_{n,k}, H^c_{n,k}) \le 4$.
    
   But this is easy, since $H_n$ can be written as the transfinite difference of 
    ${\boldsymbol \Pi^0_\xi}$ sets, so $H_n$ is obtained as the countable disjoint 
    union of sets of the form $F_{\eta} \setminus F_{\eta+1}$ with 
    $F_{\eta}, F_{\eta+1} \in {\boldsymbol \Pi^0_\xi}$, 
    and the $\al_\xi$ rank of $F_{\eta} \setminus F_{\eta+1}$ at most 4, as the sequence $(X, X, F_{\eta}, F_{\eta+1})$ shows. 
  \end{proof}
  
Now we focus our attention on finite partitions and investigate 
  the resulting rank, which we can only define  for bounded functions.
  
  \begin{definition}
    Let $f$ be a bounded Baire class $\xi$ function and let the 
    \emph{finite partition rank} of $f$ be
    \begin{equation*}
    \begin{split}
      \delta_{fin}(f) = \sup_{\varepsilon > 0} \min\bigg\{ 
        \sup_{n \le N} \alpha_\xi(H_n, H^c_n) : 
         N \in \om, H_n \in {\boldsymbol \Delta^0_{\xi + 1}}  (n \le N), \bigcup_{n \le N} H_n = X,\\ 
          H_n \cap H_m = \emptyset \;  (n, m \le N,\; n \neq m),\;    
          \exists (c_n)_{n \le N} \; 
          \bigg| f - \sum_{n \le N} c_n \cdot \chi_{H_n} \bigg| 
          \le \varepsilon \bigg\}.
    \end{split}
    \end{equation*}
  \end{definition}
  
  \begin{theorem}
    $\delta_{fin}(f) \approx \alpha_\xi(f)$ for every bounded Baire class $\xi$ function $f$.
  \end{theorem}
  \begin{proof}  
    Let $f$ be an arbitrary bounded Baire class $\xi$ function. 
    First we prove that $\delta_{fin} \lesssim \alpha_\xi(f)$. 
    For a fixed 
    $\varepsilon > 0$ let the numbers $p_n$ be defined by 
    $p_{n} = n \cdot \frac{\varepsilon}{2}$ for 
    every $n \in \mathbb{Z}$. The sets 
    $\{f \le p_n\}$ and $\{f \ge p_{n + 1}\}$ are disjoint 
    ${\boldsymbol \Pi^0_{\xi +1}}$ sets, hence they can be separated 
    by a set $A_n \in {\boldsymbol \Delta^0_{\xi + 1}}$ with 
    $\alpha_\xi(A_n, A_n^c) \le \alpha_\xi(f)$. Now let 
    $H_n = A_n \setminus A_{n - 1}$. Since $f$ is bounded, 
    $H_n = \emptyset$ for all but finitely many $n \in \omega$. Clearly, these sets form a partition, and
    $g = \sum_{n \in \mathbb{Z}} p_n \cdot \chi_{H_n}$ is $\eps$-close to $f$.
    
    We will prove in Corollary \ref{c:a_ess_lin} below that $\alpha_\xi$ is essentially linear for bounded functions. Therefore we obtain
       $\alpha_\xi(H_n, H_n^c) = \alpha_\xi(\chi_{H_n}) = 
    \alpha_\xi(\chi_{A_n} - \chi_{A_{n - 1}}) \lesssim 
    \max\{\alpha_\xi(\chi_{A_n}), \alpha_\xi(\chi_{A_{n - 1}})\}  = \max\{\alpha_\xi(A_n, A_n^c), \alpha_\xi(A_{n - 1}, A_{n-1}^c)\}
    \le \alpha_\xi(f)$, proving $\delta_{fin} \lesssim \alpha_\xi(f)$.
    
    Now we prove the other direction. Let $p < q$ be arbitrary 
    rational numbers, it is enough to prove that there is a 
    set $H \in {\boldsymbol \Delta^0_{\xi + 1}}$ separating the 
    level sets $\{f \le p\}$ and $\{f \ge q\}$ with 
    $\alpha_\xi(H, H^c) \le \delta_{fin}(f)$. Now set 
    $\varepsilon = \frac{q - p}{2}$. From the definition of 
    $\delta_{fin}$, we can find a finite partition 
    $X = H_0 \cup \dots \cup H_N$ 
    into disjoint ${\boldsymbol \Delta^0_{\xi + 1}}$ sets and 
    $c_n \in \R$ with $g = \sum_{n = 0}^N c_n \cdot \chi_{H_n}$ 
    satisfying $|f - g| < \varepsilon$ and 
    $\alpha_\xi(H_n, H_n^c) \le \delta_{fin}(f)$ for $n \le N$. 
    
    Let $A = \{n \le N : H_n \cap \{f \le p\} \neq \emptyset\}$ and 
    $H = \bigcup_{n \in A} H_n$. Clearly, $\{f \le p\} \subseteq H$.
    Moreover, no $H_n$ can intersect both $\{f \le p\}$ and $\{f \ge q\}$, 
    since $g$ is constant on $H_n$ and $|f - g| < \varepsilon = \frac{q - p}{2}$. Therefore 
    $H \cap \{f \ge q\} = \emptyset$. Using the essential linearity of  $\alpha_\xi$ for bounded functions again we obtain $\alpha_\xi(H, H^c) = \alpha_\xi(\chi_H) 
    \lesssim \max \{\alpha_\xi(\chi_{H_n}) : n \in A\} = \max \{\alpha_\xi(H_n, H^c_n) : n \in A\} \le 
    \delta_{fin}(f)$, completing the proof. 
  \end{proof}

\section{Well-behaved ranks on the Baire class $\xi$ functions}
\lab{s:pos}

  In this section we finally show that there actually exist ranks with very nice properties. Two of these ranks will answer Question \ref{q:KL} and Question \ref{q:EL}. Throughout the section, let $1 \le \xi < \om_1$ be fixed.
  
  Let $f$ be of Baire class $\xi$. Let 
  $$
    T_{f, \xi} = \{\tau':\tau' \supseteq \tau \text{ Polish}, 
    \tau' \subseteq {\boldsymbol\Sigma^0_\xi}(\tau), 
    f \in \mathcal{B}_1(\tau')\}.
  $$
  So $T_{f, \xi}$ is the set of those Polish refinements of the 
  original topology that are subsets of the 
  ${\boldsymbol\Sigma^0_\xi}$ sets turning $f$ to a Baire class $1$ 
  function. 

\begin{remark} 
\lab{r:T1}
Clearly, $T_{f, 1} = \{\tau\}$ for every Baire class 1 function $f$.
\end{remark}

In order to show that the ranks we are about to construct are well-defined, we need the following proposition.

\begin{proposition}
$T_{f, \xi} \neq \emptyset$ for every Baire class $\xi$ function $f$.
\end{proposition}

\begin{proof}
By the previous remark we may assume $\xi \ge 2$.
For every rational $p$ the level sets $\{f \le p\}$ and 
    $\{f \ge p\}$ are ${\boldsymbol \Pi^0_{\xi + 1}}$ sets, hence 
    they 
    are countable intersections of ${\boldsymbol \Sigma^0_{\xi}}$ 
    sets. In turn, these ${\boldsymbol \Sigma^0_{\xi}}$ sets are countable unions of
    sets from $\bigcup_{\eta < \xi}{\boldsymbol \Pi^0_\eta}(\tau)$. 
    Clearly, $\bigcup_{\eta < \xi}{\boldsymbol \Pi^0_\eta}(\tau) 
    \subseteq {\boldsymbol \Delta^0_\xi}$ for $\xi \ge 2$.
    By Kuratowski's theorem \cite[22.18]{K}, there exists a Polish refinement 
    $\tau' \subseteq {\boldsymbol \Sigma^0_\xi}(\tau)$ of $\tau$ 
    for which all these countable many ${\boldsymbol \Delta^0_{\xi}}$ sets 
    are in ${\boldsymbol \Delta^0_{1}(\tau')}$. 
    Then for every rational $p$ the 
    level sets are now ${\boldsymbol \Pi^0_{2}(\tau')}$ sets, and the 
    same holds for irrational numbers too, since these level sets can 
    be written as countable intersection of rational level sets, 
    proving $T_{f, \xi} \neq \emptyset$.
\end{proof}

  As in the case of limit ranks, we now define a rank on the 
  Baire class $\xi$ functions starting from an arbitrary rank on 
  the Baire class 1 functions. 
  
  \begin{definition}
    Let $\rho$ be a rank on the Baire class 1 functions. Then for a Baire class $\xi$ function 
    $f$ let  
    \begin{equation}
      \label{star_def}
      \rho_\xi^*(f) = \min_{\tau' \in T_{f, \xi}} \rho_{\tau'}(f), 
    \end{equation}
    where $\rho_{\tau'}(f)$ is just the $\rho$ rank of $f$ in the 
    $\tau'$ topology.
  \end{definition}

\begin{remark}
  From Remark \ref{r:T1} it is clear that 
  $\rho^*_1 = \rho$ for every $\rho$. 
\end{remark}

\begin{proposition}
Let $\rho$ and $\eta$ be ranks on the Baire class 1 functions. If $\rho = \eta$, or $\rho \le \eta$, or $\rho \approx \eta$, or $\rho \lesssim \eta$ then $\rho^*_\xi = \eta^*_\xi$, or $\rho^*_\xi \le \eta^*_\xi$, or $\rho^*_\xi \approx \eta^*_\xi$, or $\rho^*_\xi \lesssim \eta^*_\xi$, respectively. Moreover, the same implications hold relative to the bounded Baire class 1 functions.
\end{proposition}

\begin{proof}
The statement for $=$ and $\le$ is immediate from the definitions, and the case of $\approx$ obviously follows from the case $\lesssim$, so it suffices to prove this latter case only.
So assume $\rho \lesssim \eta$ (or $\rho \lesssim \eta$ on the bounded Baire class 1 functions).
Choose an optimal $\tau' \in T_{f, \xi}$ for $\eta$, 
    that is, $\eta^*_\xi(f) = \eta_{\tau'}(f)$. 
    Then
    $\rho^*_\xi(f) \le \rho_{\tau'}(f) \lesssim \eta_{\tau'}(f) = 
    \eta^*_\xi(f)$, completing the proof.
\end{proof}

Then the following two corollaries are immediate from Theorem \ref{alpha_le_beta_le_gamma}, and Theorem \ref{a=b=g_bounded}.

\begin{corollary}
\lab{c:a*b*g*}
$\alpha^*_\xi \le \beta^*_\xi \le \gamma^*_\xi$.
\end{corollary}

\begin{corollary}
\label{a*=b*=g*_bounded}
$\alpha^*_\xi(f) \approx \beta^*_\xi(f) \approx \gamma^*_\xi(f)$ 
for every bounded Baire class $\xi$ function $f$.
\end{corollary}

As in the case of Question \ref{q:b=g} (the case of Baire class $1$ functions), we do not know whether $\beta^*_\xi(f) \approx \gamma^*_\xi(f)$ holds for arbitrary Baire class $\xi$ functions. 

\begin{question}
 Does $\beta^*_\xi(f) \approx \gamma^*_\xi(f)$ hold for every Baire class $\xi$ function? 
\end{question}

Note that by repeating the argument of Remark \ref{alpha_not_sum} one can show that $\alpha^*_\xi$ differs from $\beta^*_\xi$ and $\gamma^*_\xi$. It is easy to see that an affirmative answer to Question \ref{q:b=g} would imply an affirmative answer to the last question, however, the other direction is not clear.

%It was shown in \cite{KL} that (if $X$ is compact then) $\beta = \gamma$ for bounded Baire class 1 functions. However, the following question is open.

%  \begin{question}
%    Let $1 \le \xi < \om_1$. Does $\beta_\xi^* \approx \gamma_\xi^*$ hold for \emph{bounded} Baire class $\xi$ functions? What if $2 \le \xi < \om_1$ 
%and we also assume that the underlying space $X$ is compact?
%  \end{question}

  \begin{theorem}
    \label{translation_invariance_star}
    If $X$ is a Polish group then the ranks $\alpha^*_\xi$, $\beta^*_\xi$ and $\gamma^*_\xi$ are 
    translation invariant. 
  \end{theorem}
  \begin{proof}
    Note first that for a Baire class $\xi$ function $f$ and 
    ${x_0} \in X$ the functions $f \circ L_{x_0}$ and 
    $f \circ R_{x_0}$ are also of Baire class $\xi$. 
    We prove the statement only for the rank $\alpha^*_\xi$, 
    because an analogous argument works for the ranks 
    $\beta^*_\xi$ and $\gamma^*_\xi$. 
    
    Let $f$ be a Baire class $\xi$ function and $x_0 \in X$, 
    first we prove that 
    $\alpha^*_\xi(f) \ge \alpha^*_\xi(f \circ R_{x_0})$. 
    Let $\tau' \in T_{f, \xi}$ be arbitrary and consider the 
    topology $\tau'' = \{ U \cdot x_0^{-1} : U \in \tau' \}$. 
    The map $\phi : x \mapsto x \cdot x_0^{-1}$ is a homeomorphism 
    between the spaces $(X, \tau')$ and $(X, \tau'')$, satisfying 
    $f(x) = (f \circ R_{x_0})(\phi(x))$. From this it is clear that 
    $\tau'' \in T_{f \circ R_{x_0}, \xi}$ 
    and since the definition of the 
    rank $\alpha$ depends only on the topology of the space, we have 
    $\alpha_{\tau'}(f) = \alpha_{\tau''}(f \circ R_{x_0})$. Since 
    $\tau' \in T_{f, \xi}$ was arbitrary, the fact that 
    $\alpha^*_\xi(f) \ge \alpha^*_\xi(f \circ R_{x_0})$ easily 
    follows. 
    
    Repeating the argument with the function 
    $f \circ R_{x_0}$ and element $x_0^{-1}$, we have 
    $\alpha^*_\xi(f \circ R_{x_0}) \ge \alpha^*_\xi(f \circ R_{x_0} 
    \circ R_{x_0^{-1}}) = \alpha^*_\xi(f)$, hence 
    $\alpha^*_\xi(f) = \alpha^*_\xi(f \circ R_{x_0})$. 
    For the function $f \circ L_{x_0}$ we can do same using the 
    topology $\tau'' = \{ x_0^{-1} \cdot U : U \in \tau' \}$ and 
    the homeomorphism $\phi : x \mapsto x_0^{-1} \cdot x$, yielding 
    $\alpha^*_\xi(f) = \alpha^*_\xi(f \circ L_{x_0})$. This 
    finishes the proof.
  \end{proof}

  \begin{theorem}
    \label{product_characteristic_star}
    If $f$ is a Baire class $\xi$ function and 
    $F \subseteq X$ is a closed set 
    then $f \cdot \chi_F$ is of Baire class $\xi$, and 
    $\alpha^*_\xi(f \cdot \chi_F) \le 1 + \alpha^*_\xi(f)$, 
    $\beta^*_\xi(f \cdot \chi_F) \le 1 + \beta^*_\xi(f)$ and 
    $\gamma^*_\xi(f \cdot \chi_F) \le 1 + \gamma^*_\xi(f)$.
  \end{theorem}
  \begin{proof}
    Examining the level sets of the function $f \cdot \chi_F$, it is 
    easy to check that it is of Baire class $\xi$. 
    
    Now let $\tau' \in T_{f, \xi}$ be arbitrary. 
    Clearly, $f \cdot \chi_F$ is of Baire class 1 with 
    respect to $\tau'$, and by 
    Proposition \ref{product_characteristic} we have 
    $\alpha_{\tau'}(f \cdot \chi_F) \le 1 + \alpha_{\tau'}(f)$
    for every $\tau' \in T_{f, \xi}$, hence 
    $\alpha^*_\xi(f \cdot \chi_F) \le 1 + \alpha^*_\xi(f)$. 
    The other two inequalities follow similarly. 
  \end{proof}

  \begin{proposition}
    \label{smaller_class_star}
    If $f$ is a Baire class $\zeta$ function with $\zeta < \xi$ 
    then $\alpha^*_\xi(f) = \beta^*_\xi(f) = \gamma^*_\xi(f) = 1$. 
  \end{proposition}
  \begin{proof}
    Using Proposition \ref{smaller_class}, it is enough to show 
    that there exists a topology $\tau' \in T_{f, \xi}$ such that 
    $f : (X, \tau') \to \R$ is continuous, and this is clear from 
    \cite[24.5]{K}. 
  \end{proof}
  
  Next we prove a useful lemma, and then investigate further properties 
  of the ranks $\alpha_\xi^*$, $\beta_\xi^*$ and $\gamma_\xi^*$.
  
  \begin{lemma}
    \label{refinement_of_topologies}
    For every $n$ let $\tau_n$ be 
    a Polish refinement of $\tau$ with $\tau_n \subseteq 
    {\boldsymbol \Sigma^0_\xi}(\tau)$. 
    Then there exists a common Polish refinement $\tau'$ of the 
    $\tau_n$'s also satisfying $\tau' \subseteq 
    {\boldsymbol \Sigma^0_\xi}(\tau)$. 
  \end{lemma}
  
  \begin{proof}
    The case $\xi=1$ is again trivial, so we may assume $\xi \ge 2$.
    Take a base $\{G^k_n : k \in \mathbb{N}\}$ for $\tau_n$. Since 
    these sets are in ${\boldsymbol \Sigma^0_\xi}(\tau)$, they can be 
    written as the countable unions of sets from 
    $\bigcup_{\eta < \xi}{\boldsymbol \Pi^0_\eta}(\tau)$. 
    Clearly, $\bigcup_{\eta < \xi}{\boldsymbol \Pi^0_\eta}(\tau) 
    \subseteq {\boldsymbol \Delta^0_\xi}$ for $\xi \ge 2$.
    As above, by Kuratowski's theorem \cite[22.18]{K}, 
    we have a Polish topology $\tau'$, for which these countably many 
    ${\boldsymbol \Delta^0_\xi}(\tau)$ sets are in ${\boldsymbol 
    \Delta^0_1}(\tau')$ satisfying $\tau' \subseteq {\boldsymbol 
    \Sigma^0_\xi}(\tau)$. This $\tau'$ works. 
  \end{proof}
  
  \begin{lemma}
    \label{in_a_finer_topology}
    If $\tau' \subseteq \tau''$ are two Polish topologies with 
    $f \in \mathcal{B}_1(\tau')$ then $f \in \mathcal{B}_1(\tau'')$, 
    moreover, $\beta_{\tau'}(f) \ge \beta_{\tau''}(f)$ and 
    $\gamma_{\tau'}(f) \ge \gamma_{\tau''}(f)$.
  \end{lemma}
  
  \begin{proof}
    To prove that $f \in \mathcal{B}_1(\tau'')$ note that 
    the level sets 
    $\{f < c \}, \{f > c \} \in {\boldsymbol \Sigma^0_2}(\tau')$, 
    hence 
    $\{f < c \}, \{f > c \} \in {\boldsymbol \Sigma^0_2}(\tau'')$,
    so $f \in \mathcal{B}_1(\tau'')$.

    Now recall the definition of the derivative defining $\beta$:
    $$
      \omega(f, x, F) = \inf \left\{ \sup_{x_1, x_2 \in U \cap F} 
      |f(x_1) - f(x_2)|: U \text{ open, } x \in U \right\}, 
    $$
    $$
      D_{f, \epsilon}(F)= \{x \in F: \omega(f,x,F) \geq 
      \epsilon\}.
    $$

    Let us now fix $f$ and $\varepsilon > 0$ and let us denote the derivative 
    $D_{f, \epsilon}$ with respect to the topology $\tau'$ by $D_{\tau'}$, and 
    with respect to the topology $\tau''$ by $D_{\tau''}$. 
    By Proposition \ref{two_derivatives} it is enough 
    to prove that  $D_{\tau''}(F) \subseteq D_{\tau'}(F)$ for every 
    closed set $F \subseteq X$. 
    
    For this it is enough to show that 
    $\omega_{\tau''}(f, x, F) \le \omega_{\tau'}(f, x, F)$ for every 
    $x \in F$  where $\omega_{\tau'}(f, x, F)$ is the oscillation 
    with respect to the topology $\tau'$. And this is clear, since in 
    the case of $\tau''$, 
    the infimum in the definition goes through more 
    open set containing $x$, hence 
    the resulting oscillation will be less.

    For the rank $\gamma$, we proceed similarly.
    First we recall the definition of $\gamma$: 
    \begin{equation*}
    \begin{split}
      \omega((f_n)_{n \in \N}, x, F) = 
        \inf_{
          \begin{subarray}{c} x \in U \\ \text{$U$ open} 
          \end{subarray}} 
        \inf_{N \in \N} 
        \sup \left\{ |f_m(y) - f_n(y)| : 
        n, m \ge N,\; y \in U \cap F \right\}, \\
	    D_{(f_n)_{n \in \N}, \varepsilon}(F) = \left\{ x \in F : 
	      \omega((f_n)_{n \in \N}, x, F) \ge \varepsilon \right\}, \\
	    \gamma(f) = 
	      \min \left\{
	        \sup_{\varepsilon > 0} \gamma((f_n)_{n \in \N}, \varepsilon) : 
	        \forall n \text{ $f_n$ is continuous and $f_n \to f$ 
	        pointwise}
	      \right\}.
    \end{split}
    \end{equation*}
  
    Let us fix a sequence $(f_n)_{n \in \N}$  
    of $\tau'$-continuous (hence also $\tau''$-continuous) functions converging pointwise to $f$, and also fix $\varepsilon > 0$.
    Let us denote the derivative 
    $D_{(f_n)_{n \in \N}, \varepsilon}$ with respect to $\tau'$ by 
    $D_{\tau'}$ and with respect to $\tau''$ by $D_{\tau''}$. 
    Again, by Proposition \ref{two_derivatives} it is enough to prove 
    that $D_{\tau''}(F) \subseteq D_{\tau'}(F)$ for every closed 
    set $F \subseteq X$. And similarly to the previous case it is 
    enough to prove that the oscillation 
    $\omega((f_n)_{n \in \N}, x, F)$ with respect to the topology 
    $\tau''$ is at most the oscillation with respect to $\tau'$, but 
    this is clear, since, as before, the infimum goes through more 
    open set in the case of $\tau''$. 
  \end{proof}

  \begin{theorem}
  \label{t:b*g*lin}
    The ranks $\beta^*_\xi$ and $\gamma^*_\xi$ are essentially linear. 
  \end{theorem}
  
  \begin{proof}
    We only consider $\beta^*_\xi$, since the proof for the rank $\gamma^*_\xi$ is completely analogous.
    
    It is easy to see that $\beta^*_\xi(cf) = \beta^*_\xi(f)$ for every $c \in \R \setminus \{0\}$, hence it suffices to show that $\beta^*_\xi$ is essentially additive.
  
    For $f$ and $g$ let $\tau_f$ 
    and $\tau_g$ be such that $\beta_{\tau_f}(f) = \beta^*_\xi(f)$ and $\beta_{\tau_g}(g) = \beta^*_\xi(g)$.
    Using Lemma \ref{refinement_of_topologies} we have a common 
    refinement $\tau'$ of $\tau_f$ and $\tau_g$ with 
    $\tau' \subseteq {\boldsymbol \Sigma^0_\xi}(\tau)$. 
    Now $f,g \in \mathcal{B}_1(\tau')$, 
    so $f+g \in \mathcal{B}_1(\tau')$, hence $\tau' \in T_{f+g, \xi}$. Therefore $\beta^*_\xi(f+g) \leq \beta_{\tau'}(f+g)$.
    By Lemma \ref{in_a_finer_topology} we have that 
    $\beta_{\tau'}(f) \leq \beta_{\tau_f}(f)$ (in fact equality 
    holds), and similarly for $g$.
    But $\beta_{\tau'}$ is additive by Theorem 
    \ref{beta_gamma_sum}, so 
    $$
      \beta^*_\xi(f+g) \leq \beta_{\tau'}(f+g) \lesssim 
      \max\{\beta_{\tau'}(f),\beta_{\tau'}(g)\} \le
      \max\{\beta_{\tau_f}(f),\beta_{\tau_g}(g)\} =
      $$
      $$
      \max\{\beta^*_\xi(f), \beta^*_\xi(g)\}.
    $$
  \end{proof}
 \begin{remark}
 \label{r:multi*}
   One can easily deduce from Theorem \ref{t:b*g*lin} that $\beta^*_\xi(f\cdot g)
   \lesssim \max\{\beta^*_\xi(f),\beta^*_\xi(g)\}$ for every $\xi<\omega_1$
   whenever $f$ and $g$ are bounded Baire class $\xi$ functions, and
   similarly for  $\gamma^*_\xi$. Again, as in the case of $\beta$ and $\gamma$, the situation is unclear for unbounded functions.
  \end{remark}
  
  \begin{question}
   \label{q:multi}
   Let $1 \le \xi < \om_1$. Are the ranks $\beta^*_\xi$ and $\gamma^*_\xi$ essentially multiplicative?
  \end{question}

  \begin{theorem}
  \lab{alpha_star_equals_alpha}
    If $f$ is a Baire class $\xi$ function then
    \begin{equation*}
      \alpha^*_\xi(f) \leq \alpha_\xi(f) \leq 2\alpha^*_\xi(f)
      \text{, hence $\alpha^*_\xi(f) \approx \alpha_\xi(f)$}.
    \end{equation*}
  \end{theorem}
  
  \begin{proof}
    For $\xi= 1$ the claim is an easy consequence of the definition of the two ranks and Corollary \ref{alpha_1_equals_alpha}. From now on, we 
    suppose that $\xi \ge 2$. 
    
    For the first inequality, for every pair of rationals $p < q$ pick 
    a sequence 
    $(F^\zeta_{p, q})_{\zeta < \alpha_\xi(f)} \subseteq 
    {\boldsymbol \Pi^0_\xi}(X)$, whose transfinite 
    difference separates the level sets $\{f \le p\}$ and 
    $\{f \ge q\}$. 
    
    Every ${\boldsymbol \Pi^0_\xi}(X)$ set is the intersection of countably many
    ${\boldsymbol \Delta^0_\xi}$ sets, hence 
    $F^\zeta_{p, q} = \bigcap_n H^\zeta_{p,q,n}$, with 
    $H^\zeta_{p,q,n} \in {\boldsymbol \Delta^0_\xi}$. 
    By Kuratowski's theorem \cite[22.18]{K}, there is a finer Polish topology 
    $\tau' \subseteq {\boldsymbol \Sigma^0_\xi}(\tau)$, 
    for which $H^\zeta_{p,q,n} \in {\boldsymbol \Delta^0_1}(\tau')$ 
    for every $p, q, n$ and $\zeta < \alpha_\xi(f)$, hence 
    $F^\zeta_{p, q} \in {\boldsymbol \Pi^0_1(\tau')}$. 
    
    This means that the level sets of $f$ can be separated by 
    transfinite differences of closed sets with respect to $\tau'$, 
     hence they can be separated by sets in 
   ${\boldsymbol \Delta^0_2}(\tau')$. Then it is easy to see that for every $c \in \R$ 
   the level sets $\{f \le c\}$ and $\{f \ge c\}$ are countable 
   intersections of ${\boldsymbol \Delta^0_2}(\tau')$ sets, hence 
   they are ${\boldsymbol \Pi^0_2}(\tau')$ sets, proving that 
    $f \in \mathcal{B}_1(\tau')$. Moreover, $\alpha_{1, \tau'}(f) \le 
    \alpha_\xi(f)$ easily follows from the construction (here $\alpha_{1, \tau'}$ is the rank $\al_1$ with respect to $\tau'$). 
    And by Corollary \ref{alpha_1_equals_alpha} we have 
    $\alpha^*_\xi \le \alpha_{\tau'}(f) \le \alpha_{1, \tau'} (f)\le 
    \alpha_\xi(f)$, proving the first inequality of the theorem. 
    
    For the second inequality, take a topology $\tau'$ with  
    $\alpha_{\tau'} (f) = \alpha^*_\xi (f)$. Again, by 
    Corollary \ref{alpha_1_equals_alpha}, we have 
    $\alpha_{1, \tau'}(f) \le 2\alpha_{\tau'}(f) = 2\alpha^*_\xi(f)$. 
    
    It remains to prove that $\alpha_\xi(f) \le \alpha_{1, \tau'}(f)$. 
    A $\tau'$-closed set is ${\boldsymbol \Pi^0_\xi}$ with respect to $\tau$.
    Therefore, if $(F_\eta)_{\eta<\zeta}$ is a decreasing continuous sequence of $\tau'$-closed sets whose transfinite difference separates 
    $\{f \le p\}$ and $\{f \ge q\}$ then the same sequence is a decreasing continuous sequence of sets from ${\boldsymbol \Pi^0_\xi}(\tau)$, proving $\alpha_\xi(f) \le \alpha_{1, \tau'}(f)$.
  \end{proof}

  \begin{corollary}
  \lab{c:a_ess_lin}
  $\alpha_\xi$ and $\alpha^*_\xi$ are essentially linear for bounded functions for every $\xi$.
  \end{corollary}
  
  \begin{proof}
  $\alpha_\xi \approx \alpha^*_\xi$ by the previous theorem, $\alpha^*_\xi \approx \beta^*_\xi$ for bounded functions by Corollary \ref{a*=b*=g*_bounded}, and $\beta^*_\xi$ is essentially linear by Theorem \ref{t:b*g*lin}.
  \end{proof}
 From Corollary \ref{c:r_stepf} we can obtain the appropriate statement for the ranks $\alpha^*_\xi, \beta^*_\xi$ and $\gamma^*_\xi$.
  \begin{proposition}
    \label{c:r*_stepf} If $f=\sum_{i = 1}^{n} c_i \chi_{A_i}$, where the $A_i$'s are disjoint $\mathbf{\Delta}^0_{\xi+1}$ sets covering $X$ and the $c_i$'s are distinct then 
\[\alpha^*_\xi(f) \approx\max_i \{\alpha^*_\xi(\chi_{A_i})\},\] and similarly for $\beta^*_\xi$ and $\gamma^*_\xi$.
   \end{proposition}
   \begin{proof}
   The additivity of $\alpha^*_\xi$ implies $\alpha^*_\xi(f) \lesssim \max_i \{\alpha^*_\xi(\chi_{A_i})\}$. 
   For the other inequality let  
   $\tau'$ be a topology for which $f$ is Baire class $1$. Then the characteristic functions $\chi_{A_i}$
are also Baire class $1$, and hence by Corollary \ref{c:r_stepf} 
 we obtain $\alpha_{\tau'}(f) \approx \max_i \{\alpha_{\tau'}(\chi_{A_i})\}$. But by the definition of $\alpha^*_\xi$ for every
such topology $\alpha^*_\xi(\chi_{A_i}) \leq \alpha_{\tau'}(\chi_{A_i})$, therefore $\max_i \{ \alpha^*_\xi(\chi_{A_i}) \} \leq \max_i \{ \alpha_{\tau'}(\chi_{A_i}) \} \approx \alpha_{\tau'}(f)$. Then choosing $\tau'$ so that $\alpha_{\tau'}(f) = \alpha^*_\xi(f)$ the proof is complete. 
	
	\end{proof}

  \begin{theorem}
    \label{star_not_bounded}
    The ranks $\alpha^*_\xi$, $\beta^*_\xi$ and $\gamma^*_\xi$
    are unbounded in $\om_1$. Moreover,
for every non-empty perfect set $P \subseteq X$ and ordinal 
    $\zeta < \omega_1$ there exists a characteristic function $\chi_A \in \mathcal{B}_\xi(X)$ 
     with $A \subseteq P$ such that
    $\alpha^*_\xi(\chi_A), \beta^*_\xi(\chi_A)$, $\gamma^*_\xi(\chi_A) \ge \zeta$. 
  \end{theorem}
  
  \begin{proof}
In order to prove the theorem, by Corollary \ref{c:a*b*g*} it suffices to prove the statement for $\alpha^*_\xi$.
Moreover, instead of $\alpha^*_\xi(\chi_A) \ge \zeta$ it suffices to obtain $\alpha^*_\xi(\chi_A) \gtrsim \zeta$. And this is clear from Theorem \ref{alpha_star_equals_alpha} and Corollary \ref{alpha_unbounded_on_perfect}.
  \end{proof}

  \begin{proposition}
  \lab{p:b_unif}
    If $f_n, f$ are Baire class $\xi$ functions and $f_n \to f$ 
    uniformly then $\beta^*_\xi(f) \le \sup_n \beta^*_\xi(f_n)$. 
  \end{proposition}
  \begin{proof}
    For every $n$ let $\tau_n \in T_{f_n, \xi}$ with 
    $\beta_{\tau_n}(f_n) = \beta^*_\xi(f_n)$. 
    Using Lemma 
    \ref{refinement_of_topologies}, let $\tau'$ be 
    their common refinement satisfying 
    $\tau' \subseteq {\boldsymbol \Sigma^0_\xi}(\tau)$, where $\tau$ 
    is the original topology. Note that $f_n \in \iB_1(\tau')$ for every $n$, and the Baire class 1 functions are closed under uniform limits \cite[24.4]{K}, hence $\tau' \in T_{f, \xi}$. Then by Proposition \ref{beta_uniform_limit} 
    and Lemma \ref{in_a_finer_topology} we have 
    $$
      \beta^*_\xi(f) \le \beta_{\tau'}(f) \le 
        \sup_n \beta_{\tau'}(f_n) \le 
        \sup_n \beta_{\tau_n}(f_n) = \sup_n \beta^*_\xi(f_n).
    $$ 
  \end{proof}
  
  \begin{proposition}
   \lab{p:a_unif}
    If $f_n, f$ are Baire class $\xi$ functions and $f_n \to f$ 
    uniformly then 
    $\alpha^*_\xi(f) \lesssim \sup_n \alpha^*_\xi(f_n)$ and $\gamma^*_\xi(f) \lesssim \sup_n \gamma^*_\xi(f_n)$.
  \end{proposition}
  \begin{proof}
    Repeat the previous argument but apply Proposition 
    \ref{alpha_uniform_limit} and Proposition \ref{gamma_uniform_limit} instead of Proposition 
    \ref{beta_uniform_limit}.
  \end{proof}

\section{Uniqueness of the ranks}
\lab{s:uni}

As we have seen, the natural unbounded ranks defined on the Baire class $\xi$
functions essentially coincide on the bounded functions. Now we will formulate a
general theorem which states that if a rank on the bounded functions has certain natural properties then
it must agree with the ranks defined above. Because of some not completely clear technical difficulties we only work out the details in the Baire class $1$ case.

The main reason why we treat this result separately and did not use it to prove that the ranks considered so far all agree for bounded functions is the following. So far, formally, a rank was simply a map defined on a set of functions. Now we slightly modify this concept: in this section a rank will be a family of maps $\rho = \{\rho^{(X, \tau)}\}_{(X, \tau) \textrm{ Polish}}$, where $\rho^{(X, \tau)}$ is a rank on the Baire class $1$ functions defined on the Polish space $(X, \tau)$. However, since there is no danger of confusion, we will abuse notation and will simply continue to use $\rho$. Notice that the ranks $\alpha,\beta$ and $\gamma$ can naturally be viewed this way.

\begin{theorem}
 \lab{t:uni}
 Let $\rho$ be a rank on the bounded Baire class $1$ functions. Suppose that
$\rho$ has the following properties for every $A \in \mathbf{\Delta}^0_2$ and Baire class $1$ functions $f$ and $f_n$:
 \begin{enumerate}
  \item $\rho(\chi_{A}) \approx
\alpha_1 (A,A^c)\\ (\approx \alpha(A, A^c) \approx \alpha(\chi_{A})\approx \beta(\chi_{A}) \approx
\gamma(\chi_A)$, that is, the rank of $A$ is essentially its complexity in the difference hierarchy),
  \item $\rho$ is essentially linear,
  \item if $f_n \to f$ uniformly then $\rho(f) \lesssim \sup_{n} \rho(f_n)$,
  \item if $h:\R \to \R$ is a Lipschitz function then $\rho(h \circ f) \lesssim
\rho(f)$,
   \item \label{5}
    if $f$ is defined on the Polish space $X$ and $Y
\subset X$ is Polish (or equivalently, $\mathbf{\Pi}^0_2(X)$, see e.g. \cite[3.11]{K}) then $\rho(f|_Y)
\lesssim \rho(f)$.
   
 \end{enumerate}
Then $\rho \approx \alpha$ for bounded Baire class $1$ functions.
\end{theorem}

Property $(\ref{5})$ is probably the most ad hoc among the
conditions, however it is easy to see that it holds for ranks $\alpha,\beta$
and $\gamma$:

\begin{lemma}
\label{l:alpha_megsz} 
Let $X,Y$ be Polish spaces with $Y \subset X$ and $f$ be a bounded Baire class $1$
function on $X$. 
Then $\alpha(f|_Y) \lesssim \alpha(f)$, and hence similarly for $\beta$ and $\gamma$.
\end{lemma}

\begin{proof}
 Using Corollary \ref{alpha_1_equals_alpha}, it is enough to prove the 
 lemma for $\alpha_1$. 
 By the definition of the rank $\alpha_1$, if $p<q$ are rational numbers then
 there exists a $\mathbf{\Delta}^0_2(X)$ set $A$ so that $\alpha_1(A,A^c) \leq
\alpha_1(f)$ and $A$ separates $\{f\leq p\}$ and $\{f \geq q\}$. Clearly, $A \cap
Y$ separates the sets  $\{f|_Y\leq p\}$ and $\{f|_Y \geq q\}$. So it is enough
to show that $\alpha_{1, Y}(A\cap Y, A^c \cap Y) \leq \alpha_1(A,A^c)$. 

Now, 
there exists a sequence of closed sets $(F_\eta)_{\eta<\alpha_1(A,A^c)}$ so that 
\[A= \bigcup_{\begin{subarray}{c} \eta < \alpha_1 (A,A^c) \\ \eta \text{ even}
\end{subarray}} 
 (F_{\eta} \setminus F_{\eta + 1}).\] But the sets
$(F_\eta \cap Y)_{\eta<\alpha_1(A,A^c)}$ witness that $\alpha_{1, Y}(A\cap Y, A^c \cap
Y) \leq \alpha_1(A,A^c)$, so we are done.
\end{proof}

\begin{proof}[Proof of Theorem \ref{t:uni}] We split the proof of the theorem into two easy lemmas.

\begin{lemma}
If $f=\sum_{i=1}^n c_i \chi_{A_i}$ where the 
$A_i$'s are disjoint $\mathbf{\Delta}^0_2$ sets covering the underlying space $X$ and the $c_i$'s are distinct then $\rho(f) \approx
\alpha(f)$.
\end{lemma}
 \begin{proof}
 
 By the essential linearity of $\rho$ clearly 
 \[\rho(f) \lesssim \max_i \rho(\chi_{A_i}).\]
 Now let $0 \leq j \leq n$ be fixed and $h: \R \to \R$ be Lipschitz so that $h(c_i)=0$ for $i \not = j$ and $h(c_j)=1$.
 Then \[\rho(\chi_{A_j}) = \rho(h \circ f) \lesssim \rho (f)\] by Property
$(4)$, so 
 
 \[\rho(f) \approx \max_i \rho(\chi_{A_i}).\]
 
 Using Corollary \ref{c:r_stepf} and Property $(1)$ we obtain that
$\alpha$ and $\rho$ essentially agree on step functions.
 \end{proof}
 Now let $f$ be an arbitrary bounded Baire class $1$ function. Then by Lemma
\ref{f_n_uniformly_conv} and Proposition \ref{alpha_uniform_limit} there exists
a sequence of step functions $f_n$ converging uniformly to $f$ so that
$\alpha(f)  \approx \sup_{n} \alpha(f_n)$. Hence, by Property $(3)$ and the previous lemma, 
 \[\rho(f) \lesssim \sup_{n} \rho(f_n) \approx \sup_{n} \alpha(f_n) \approx
\alpha(f).\]
 
 Hence, interchanging the role of $\al$ and $\rho$ in the above argument, in order to prove $\rho(f) \approx \alpha(f)$ it is enough to construct a sequence $f_n$ of step functions converging uniformly to $f$ so that 
\begin{equation}
\label{e:v}
\sup_n \rho(f_n) \lesssim \rho(f).
\end{equation} 
 
 The construction goes similarly to that of  Lemma
\ref{f_n_uniformly_conv}, but we need an additional step.
 
  \begin{lemma}
  \label{l:szep}
  Suppose that $f$ is a bounded Baire class $1$ function on the Polish space
$X$ and $p,q \in \R$ 
  with $p<q$. Then there exists a set $H \in \mathbf{\Delta}^0_{2}(X)$ so that
$\rho(\chi_H) \lesssim \rho(f)$ and $H$ separates the sets $\{f \leq p\}$ and
$\{f \geq q\}$.  
 \end{lemma}
 
 \begin{proof}
  Let $h: \R \to \R$ be Lipschitz so that $h|_{(-\infty,p]} \equiv 0$ and
$h|_{[q,\infty)} \equiv 1$, and $f_1=h \circ f$. Property $(4)$ ensures that
\begin{equation}
\lab{e:1}
\rho(f_1) \lesssim \rho(f).
\end{equation}
  
  Let $Y=\{f \leq p\}\cup\{f \geq q\}$ and $f_2=f_1|_Y$. Clearly, $f_2$ is a
step function on the Polish space $Y$ (note that $Y$ is $\mathbf{\Pi}^0_2(X)$),
hence by the previous lemma and Property $(5)$ we obtain
\begin{equation}
\lab{e:2}
\alpha(f_2) \approx \rho(f_2) \lesssim \rho(f_1).
\end{equation}
  
  In particular, $\{f_2 \leq 0\}$ and $\{f_2 \geq 1\}$ can be separated by a
$\mathbf{\Delta}^0_2(Y)$ set $H'$ so that 
  \[H'=\bigcup_{
        \begin{subarray}{c} \eta < \lambda \\ \eta 
          \text{ even} \end{subarray}} 
        (F'_{\eta} 
      \setminus F'_{\eta + 1})\]
     for some $F'_\eta \in \mathbf{\Pi}^0_1(Y)$  and
\begin{equation}
\lab{e:3}
\lambda \lesssim \alpha(f_2), 
\end{equation}
using Corollary \ref{alpha_1_equals_alpha}.
   
    Now let $F_\eta$ be the closure of $F'_\eta$ in $X$ and
    
\[
H=
\bigcup_{
        \begin{subarray}{c} \eta < \lambda \\ \eta 
          \text{ even} \end{subarray}} 
        (F_{\eta} 
      \setminus F_{\eta + 1}).
\]
      
      Then $H$ is a $\mathbf{\Delta}^0_2(X)$ set, and by Property (1), Corollary \ref{alpha_1_equals_alpha}, \eqref{e:3}, \eqref{e:2} and \eqref{e:1} we obtain 
\[
\rho(\chi_H) \approx \alpha(\chi_H) \le \la \lesssim \alpha(f_2) \approx \rho(f_2) \lesssim \rho(f_1) \lesssim \rho(f).
\]
      
      Moreover, \[H \cap Y =\bigcup_{
        \begin{subarray}{c} \eta < \lambda \\ \eta 
          \text{ even} \end{subarray}} 
        (F_{\eta} 
      \cap F^c_{\eta + 1} \cap Y)=\bigcup_{
        \begin{subarray}{c} \eta < \lambda \\ \eta 
          \text{ even} \end{subarray}} 
        (F'_{\eta} 
      \cap F'^c_{\eta + 1} \cap Y)= H' \cap Y.\]

      Since $H'$ separates $\{f_2 \leq 0\}$ and $\{f_2 \geq 1\}$, and it is easy to see 
that $\{f \leq p \} \subset \{f_2 \leq 0\} \subset Y$
and analogously for $\{f \geq q\}$, we obtain that $H$ separates $\{f \leq p\}$ and
$\{f \geq q\}$, which completes the proof.
\end{proof}

      Now we complete the proof by constructing a sequence $f_n$ converging uniformly to $f$ and satisfying \eqref{e:v}. We basically repeat the proof of Lemma \ref{f_n_uniformly_conv}. 
      If $f$ is a constant function then $f_n = f$ works. So suppose that $f$ is not constant, and let $p_{n, k} = k / 2^n$ for all $k \in \Z$ and $n \in \N$ so that
$\inf(f) \leq p_{n, k} \leq \sup(f)$. By the boundedness of $f$ there are just
finitely many $p_{n,k}$'s for a fixed $n$. The 
      level sets $\{f \le p_{n, k}\}$ and $\{f \ge p_{n, k + 1}\}$ are 
      disjoint ${\boldsymbol \Pi^0_2}$ sets, hence by the previous lemma they can be separated 
      by a $H_{n, k} \in {\boldsymbol \Delta^0_2}$
      so that $\rho(\chi_{H_{n,k}}) \lesssim \rho(f)$. 
             Set 
      \begin{equation*}
        f_n = \sum_{k} p_{n, k} \cdot 
          (\chi_{H_{n, k + 1}}-\chi_{ H_{n, k}}).
      \end{equation*}
      Clearly, $f_n \to f$ uniformly. Now, for every $n$
\[
\rho(f_n) = \rho \left(\sum_{k} p_{n, k} \cdot 
          (\chi_{H_{n, k + 1}}-\chi_{ H_{n, k}})\right) \lesssim \max_{k}
\rho(\chi_{H_{n,k}}) \lesssim \rho(f)
\]
by the essential linearity of $\rho$, which finishes the proof
of the theorem. 
    \end{proof}       
      
\begin{remark}
We claim that if the range of our functions is the triadic Cantor set 
$\mathcal{C} \subseteq \R$ instead of $\R$ then we can drop Property $(5)$ in Theorem \ref{t:uni}. In order to see this, we show that Lemma \ref{l:szep} can be proved without using Property $(5)$. 
Let $f : X \to \mathcal{C}$ and $p, q \in \mathcal{C}$ be as in the lemma. 
Let $A \in \boldsymbol{\Delta}^0_1(\mathcal{C})$ with $\{x \in \mathcal{C} : x \le p\} \subseteq A \subseteq \{x \in \mathcal{C} : x \ge q\}^c$. 
Then $h = \chi_A$ is Lipschitz, since $A$ and $A^c$ are two 
disjoint compact subsets, hence their distance is positive.  This 
implies $\rho(h \circ f) \lesssim \rho(f)$ 
by Property $(4)$. 
Let $H = f^{-1}(A)$, then $H \in \boldsymbol{\Delta}^0_2(X)$, 
$\{f \le p\} \subseteq H \subseteq \{f \ge q\}^c$ and 
$\rho(\chi_H) = \rho(h \circ f) \lesssim \rho(f)$, since 
$\chi_H = h \circ f$. This proves our claim.
\end{remark}

      \begin{question}
       Does there exist a rank $\rho$ with Properties $(1)-(4)$, so that $\rho
\not \approx \alpha$?
      \end{question}
      
Now we very briefly discuss the Baire class $\xi$ case. It is not hard to check that if the family of ranks is defined not only on functions on the Polish spaces, but also on functions on all subsets (or just Borel or
$\mathbf{\Pi}^0_{\xi+1}$ subsets) of Polish spaces, and Property (5) is modified accordingly, then a result analogous to Theorem \ref{t:uni} holds. However, the following question, where the ranks are only defined on functions on the Polish spaces is more natural.
      
\begin{question}
        Let $\rho$ be rank on the bounded Baire
class $\xi$ functions (defined on Polish spaces). Suppose that $\rho$ has the following properties:
 \begin{enumerate}
  \item if $A \in \mathbf{\Delta}^0_{\xi+1}(X)$ then $\rho(\chi_A) \approx
\alpha_\xi(\chi_A)$,
  \item $\rho$ is essentially linear,
  \item if $f_n \to f$ uniformly then $\rho(f) \lesssim \sup_{n} \rho(f_n)$,
  \item if $h:\R \to \R$ is a Lipschitz function then $\rho(h \circ f) \lesssim \rho(f)$,
   \item if $H \in \mathbf{\Pi}^0_2(X)$ then $\rho(f|_H) \lesssim \rho(f)$.
   
 \end{enumerate}
Does this imply that $\rho \approx \alpha$ for bounded Baire class $\xi$ functions?
\end{question}

\section{Conclusion}

First we answered Question \ref{q:KL2} affirmatively by showing that the underlying compact metric space in the theory of Kechris and Louveau can be replaced by an arbitrary Polish space.

Then, after proving that certain very natural attempts surprisingly result in ranks that are bounded in $\om_1$, we have defined three ranks on the Baire class $\xi$ functions, 
  $\alpha^*_\xi \le \beta^*_\xi \le \gamma^*_\xi$, 
  corresponding to the three ranks on the Baire class 1 functions 
  investigated by Kechris and Louveau. All the other ranks for which we could prove unboundedness, namely $\alpha_\xi$ and $\delta_{fin}$ 
  defined on the bounded Baire class $\xi$ functions, essentially agree with $\alpha^*_\xi$. 
  (It is unclear whether $\alpha'_\xi$ is unbounded, see the next section of Open problems.) 
  
  If we consider the ranks of sets, i.e., the ranks of characteristic 
  functions, or more generally, the ranks of bounded functions, 
  then in addition $\alpha^*_\xi \approx \beta^*_\xi \approx \gamma^*_\xi$ holds, hence all ranks are essentially the same for bounded functions!
  We also have a general result (only spelled out in the Baire $1$ case) that all ranks satisfying certain natural requirements agree on the bounded functions. Moreover, the rank of a step function $\sum_{i=1}^n c_i \chi_{A_i}$ (where the $A_i$'s form a partition and the $c_i$'s are distinct) is the maximum of the ranks of the $\chi_{A_i}$'s. 
  
 %(It is perhaps the most important open problem if $\beta^*_\xi \approx \gamma^*_\xi$ also holds for bounded functions.)
  
  We were able to prove most of the known properties of the ranks 
  on the Baire class 1 functions for
  $\alpha^*_\xi$, $\beta^*_\xi$ and $\gamma^*_\xi$. 
  All three ranks are translation invariant and unbounded in $\omega_1$.
  The ranks $\beta^*_\xi$ and $\gamma^*_\xi$ are essentially linear, 
  while $\alpha^*_\xi$ is not. 
  The ranks $\alpha^*_\xi, \beta^*_\xi$ and $\gamma^*_\xi$ behave nicely under uniform limits.

 %(While this is open for $\gamma^*_\xi$.) 

This may well be considered as an affirmative answer to the (slightly vague) Question \ref{q:KL}. Moreover, we have the following.

\begin{corollary}
The rank $\beta^*_\xi$ (or $\gamma^*_\xi$) provides an affirmative answer to Question \ref{q:EL}.
\end{corollary}

\begin{proof}
The proofs of the requirements listed in the question can be found in 
\begin{itemize}
 \item Theorem \ref{star_not_bounded},
 \item Theorem \ref{translation_invariance_star},
 \item Theorem \ref{t:b*g*lin},
 \item Theorem \ref{product_characteristic_star} (note that $1+\eta \lesssim \eta$ for every $\eta$),
\end{itemize}
respectively.
\end{proof}

Then, by considering the proof of \cite[Theorem 6.2]{EL} and replacing the class of Borel functions by $\iB_\xi$, the Borel class by the rank $\beta^*_\xi$ and the functions $\chi_{B_\al}$ by functions supported in $P_\al$ with $\beta^*_\xi$ rank at least $\alpha$ we obtain the following.

\begin{corollary}
For every $2 \le \xi < \om_1$ the solvability cardinal $\mathop{sc}(\iB_\xi) \ge \om_2$, hence under the Continuum Hypothesis $\mathop{sc}(\iB_\xi) = \om_2 = (2^\om)^+$.
\end{corollary}

\section{Open problems}
\lab{s:open}

In this last section we collect the open problems of the paper.

Throughout the paper we almost always considered only the relations $\approx$ and $\lesssim$. It would be interesting to know which statements remain true using $=$ and $\le$ instead. 

\begin{question}
 Let $\rho$ and $\rho'$ be two of the ranks defined in this paper for which $\rho \lesssim \rho'$ holds. Is it true that $\rho \leq \rho'$? 
\end{question}

We have shown in Theorem \ref{t:a'} that if $1 \le \xi < \om_1$ and $f$ is a \emph{characteristic} Baire class $\xi$ function then the linearized separation rank $\alpha'_\xi(f) \le 2$.

  \begin{question}
    Is the linearized separation rank $\alpha'_\xi$ unbounded in $\om_1$ for the Baire class 
    $\xi$ functions?
  \end{question}
  Actually, we do not even know the answer when $\xi = 1$.

The following question is very closely related to this.

 \begin{question}
    Let $1 \le \xi < \om_1$ and let $f_n$ and $f$ be Baire class $\xi$ functions such that $f_n \to f$ uniformly. Does this imply that $\alpha'_\xi (f) \lesssim \sup_n \alpha'_\xi (f_n)$?
  \end{question}
  
  As mentioned above, an affirmative answer to this question would provide a negative answer to the previous one.
  
  Recall that a rank $\rho$ is \emph{essentially multiplicative} if $\rho (f\cdot g)
   \lesssim \max\{\rho(f),\rho(g)\}$ for every $f$ and $g$. Remarks \ref{multiplicativity} and \ref{r:multi*} indicate that the ranks $\beta$, $\gamma$, $\beta^*_\xi$ and $\gamma^*_\xi$ are essentially multiplicative on the \emph{bounded} functions from the appropriate Baire classes.
 
 \begin{question}
   Let $1 \le \xi < \om_1$. Are the ranks $\beta$, $\gamma$, $\beta^*_\xi$ and $\gamma^*_\xi$ essentially multiplicative?
\end{question}

We have shown in Theorem \ref{t:limit_bdd} that the limit ranks are bounded by
$\omega$, but do not know whether this is optimal.
  \begin{question}
    Is there an $n \in \om$ such that $\overline{\gamma}_2 \le n$? 
    If yes, which is the smallest such $n$?    
  \end{question}
We have seen that for every $1 \le \xi < \om_1$ we have $\beta^*_\xi \approx \gamma^*_\xi$ on the bounded Baire class $\xi$
functions (even on non-compact Polish spaces), but $\alpha^*_\xi \not \approx \beta^*_\xi$ for
arbitrary Baire class $\xi$ functions. So the following question is natural.
  
\begin{question}
   Let $1 \le \xi < \om_1$. Does $\beta^*_\xi \approx \gamma^*_\xi$ hold for arbitrary Baire class $\xi$ functions? 
\end{question}

We believe that an affirmative answer might help extend Theorem \ref{t:uni}
to the unbounded case. 

Our next questions concern the uniqueness of ranks.

\begin{question}
Does there exist a rank $\rho$ with Properties $(1)-(4)$ of Theorem
\ref{t:uni} so that $\rho \not \approx \alpha$ on bounded Baire class $1$
functions?
\end{question}
   
\begin{question}
        Let $\rho$ be rank on the bounded Baire
class $\xi$ functions (defined on Polish spaces). Suppose that $\rho$ has the following properties:
 \begin{enumerate}
  \item if $A \in \mathbf{\Delta}^0_{\xi+1}(X)$ then $\rho(\chi_A) \approx
\alpha_\xi(\chi_A)$,
  \item $\rho$ is essentially linear,
  \item if $f_n \to f$ uniformly then $\rho(f) \lesssim \sup_{n} \rho(f_n)$,
  \item if $h:\R \to \R$ is a Lipschitz function then $\rho(h \circ f) \lesssim \rho(f)$,
   \item if $H \in \mathbf{\Pi}^0_2(X)$ then $\rho(f|_H) \lesssim \rho(f)$.
   
 \end{enumerate}
Does this imply that $\rho \approx \alpha$ for bounded Baire class $\xi$ functions?
\end{question}

\begin{question}
The fourth chapter of \cite{KL} discusses two more ranks on the bounded Baire class 1 functions that turn out to be essentially equivalent to $\alpha, \beta$ and $\gamma$. Is there a well-behaved generalization of these theories to the Baire class $\xi$ case? 
\end{question}

\subsection*{Acknowledgments}
We are greatly indebted to M. Laczkovich for numerous helpful discussions. We are also grateful to P. Dodos, A. Kechris, A. Louveau and S. Solecki for many valuable remarks.

\newpage

M\'arton Elekes

Alfr\'ed R\'enyi Institute of Mathematics

Hungarian Academy of Sciences

P.O. Box 127, H-1364 Budapest, Hungary

elekes.marton@renyi.mta.hu

www.renyi.hu/ $\tilde{}$ emarci

and

E\"otv\"os Lor\'and University

Department of Analysis

P\'az\-m\'any P. s. 1/c, H-1117, Budapest, Hungary

\bigskip

Viktor Kiss

E\"otv\"os Lor\'and University

Department of Analysis

P\'az\-m\'any P. s. 1/c, H-1117, Budapest, Hungary

kivi@cs.elte.hu

\bigskip

Zolt\'an Vidny\'anszky

Alfr\'ed R\'enyi Institute of Mathematics

Hungarian Academy of Sciences

P.O. Box 127, H-1364 Budapest, Hungary

vidnyanszky.zoltan@renyi.mta.hu

www.renyi.hu/ $\tilde{}$ vidnyanz

and

E\"otv\"os Lor\'and University

Department of Analysis

P\'az\-m\'any P. s. 1/c, H-1117, Budapest, Hungary

\end{document}